\documentclass[11pt, a4paper]{amsart}

\usepackage[utf8]{inputenc}

\usepackage{xcolor}
\definecolor{darkgreen}{rgb}{0,0.5,0}
\usepackage[
        colorlinks, citecolor=darkgreen,
        backref,
        pdfauthor={S. Gajovi\'c, J.S. M\"uller},
        pdftitle={Computing $p$-adic heights on hyperelliptic curves}
]{hyperref}
\usepackage[hyphenbreaks]{breakurl}

\usepackage{fancyhdr}
\usepackage{tikz,enumerate,caption,stmaryrd,amsfonts,amssymb,systeme,comment}
\usetikzlibrary{matrix,positioning,arrows,shapes,chains,calc,automata}
\usetikzlibrary{decorations.pathreplacing}
\usepackage{amssymb}
\usepackage{amsmath}
\usepackage{mathtools}
\usepackage{amsthm}
\usepackage{tikz-cd}
\usepackage[OT2,T1]{fontenc}
\usepackage{url}
\DeclareSymbolFont{cyrletters}{OT2}{wncyr}{m}{n}
\DeclareMathSymbol{\Sha}{\mathalpha}{cyrletters}{"58}

\usepackage{cleveref}
\crefformat{section}{§#2#1#3}
\usepackage{colonequals}
\newcommand{\R}{\mathbb{R}}
\newcommand{\PP}{\mathbb{P}}
\newcommand{\Q}{\mathbb{Q}}

\newcommand{\C}{\mathbb{C}}
\newcommand{\calC}{\mathcal{C}}
\newcommand{\calU}{\mathcal{U}}

\newcommand{\calD}{\mathcal{D}}
\newcommand{\Z}{\mathbb{Z}}

\newcommand{\F}{\mathbb{F}}

\renewcommand{\O}{\mathcal{O}}
\newcommand{\p}{\mathfrak{p}}
\newcommand{\q}{\mathfrak{q}}
\newcommand{\hdr}{\operatorname{H^1_{dR}}}
\newcommand{\hrig}{\operatorname{H^1_{rig}}}
\newcommand{\Span}{\operatorname{Span}}

\newcommand{\hmw}{\operatorname{H^1_{MW}}}

\newcommand{\tors}{\operatorname{tors}}

\newcommand{\Div}{\operatorname{Div}}
\newcommand{\dv}{\operatorname{div}}

\newcommand{\Res}{\operatorname{Res}}
\newcommand{\Reg}{\operatorname{Reg}}

\newcommand{\Spec}{\operatorname{Spec}}

\newcommand{\supp}{\operatorname{supp}}
\newcommand{\ord}{\operatorname{ord}}
\newcommand{\rk}{\operatorname{rk}}
\newcommand{\Frob}{\operatorname{Frob}}

\renewcommand{\O}{\mathcal{O}}

\newtheorem{theorem}{Theorem}[section]
\newtheorem{proposition}[theorem]{Proposition}
\newtheorem{corollary}[theorem]{Corollary}
\newtheorem{lemma}[theorem]{Lemma}
\newtheorem{cond}{Condition}
\newtheorem{ass}{Assumption}
\newtheorem{conj}[theorem]{Conjecture}
\newtheorem{alg}[theorem]{Algorithm}
\theoremstyle{remark}
\newtheorem{remark}[theorem]{Remark}
\newtheorem{definition}[theorem]{Definition}
\newtheorem{example}[theorem]{Example}

\numberwithin{equation}{section}

\newcommand{\sg}{{}}
\newcommand{\sm}{{}}

\author[Stevan Gajovi\'c]{Stevan Gajovi\'c}
\address{Stevan Gajovi\'c, Charles University, Faculty of Mathematics and Physics, Department of Algebra, Sokolov\-sk\' a 83, 186~75 Praha~8, Czech Republic}
\email{gajovic@karlin.mff.cuni.cz}

\author[J.~Steffen M\"uller]{J.~Steffen M\"uller}
\address{J.~Steffen M\"uller, University of Groningen, Faculty of Science and Engineering, Algebra, Nijenborgh~9, 9747~AG Groningen, Netherlands}
\email{steffen.muller@rug.nl} 
\title[Computing $p$-adic heights on hyperelliptic curves]{Computing $p$-adic heights on
hyperelliptic curves}
\date{}

\begin{document}

\maketitle

\begin{abstract}
We describe an algorithm to compute the local Coleman--Gross $p$-adic height 
  at $p$ on a hyperelliptic curve. Previously, this was only possible using an algorithm
due to Balakrishnan and Besser, which was limited to odd degree. While we follow
  their general strategy, our
algorithm is significantly faster and simpler and works for both odd and even degree. 
  We discuss a precision analysis and an implementation in SageMath.
  Our work has several applications, also discussed in this article. 
  These include various versions of the 
  quadratic Chabauty method, and numerical evidence 
  for a $p$-adic version of the conjecture of Birch and Swinnerton--Dyer
  in cases where this was not previously possible.
\end{abstract}



\section{Introduction} \label{sec:introduction}

Algorithms for the computation of $p$-adic heights on Jacobian varieties 
curves have recently played a crucial
part in explicit methods for the computation of integral and
rational points on curves via the quadratic Chabauty
method. Such heights also appear in
$p$-adic analogues of the conjecture of Birch and Swinnerton--Dyer.
In this work, we restrict to the case of the Jacobian $J$
of a hyperelliptic curve $$X\colon y^2=f(x)\,,\quad f\in
\Z[x]\;\;\text{squarefree}$$
over $\Q$ and a prime $p$ of good reduction for $X$. 
Denote by
$h\colon J(\Q)\times J(\Q)\to \Q_p$ the $p$-adic height pairing constructed
by Coleman and Gross~\cite{Coleman-Gross-Heights} (depending on various
choices).
Then $h$ is bilinear, and is a sum of local height pairings $h_v(\cdot,\cdot)$
for each finite prime $v\in \Z$. These are defined on pairs of divisors on
$X\otimes \Q_v$ of degree~0 with disjoint support, and the construction 
depends on whether $v\ne p$ or $v=p$.
The former are defined in terms of intersection theory on arithmetic
surfaces and can be computed using algorithms discussed in \cite{Hol12,
Mul14, Raymond-David-Steffen}.
Our work focuses on the latter; these are given
in terms of $p$-adic analysis. 
More precisely, one has $h_p(D_1,D_2)=\int_{D_2}\omega_{D_1}$, where
$\omega_{D_1}$ is a certain differential of the third kind on $X\otimes \Q_p$ whose residue divisor is $D_1$, and  the integral is a Coleman integral
\cite{ColemanI}, \cite{Coleman-deShalit}.

When $X$ has an odd degree model over $\Q_p$, one can compute $h_p$
using work of 
Balakrishnan and Besser~\cite{BBHeights}. 
Based on their strategy, we describe  in Sections~\ref{sec:Algorithm-computations}, \ref{sec:Even-two-infinities}  and~\ref{sec:affine}
an algorithm  to compute $h_p$ that does not
require $X$ to have an odd degree model over $\Q_p$, and
is much simpler and faster than the one from~\cite{BBHeights}  (see the
timings in~\S\ref{subsec:timings}).  
We use it 
\begin{itemize}
  \item to obtain a vast speed-up of the quadratic Chabauty method for
rational points in some situations (see~\S\ref{subsec:qcrat});
\item    to develop a new
quadratic Chabauty method for integral points on even degree
hyperelliptic curves over number fields that is simpler than previous
instances of this method
(see~\S\ref{subsec:intro-QC-details} and~\cite{LinQC}) and crucially relies
on our new algorithm;
    \item to give numerical evidence for  the
$p$-adic Birch and Swinnerton--Dyer conjecture from~\cite{BMS16}
in cases where this was not previously possible
(see~\S\ref{subsec:pbsd} and Section~\ref{S:pbsd}).
\end{itemize}

We now sketch our method.
Assume that $X$ is given by an even degree model (we will see in Section \ref{sec:Algorithm-computations} that this is sufficient), and, for simplicity, that $f$ is monic. Denote by $\infty_-$ and
$\infty_+$ the two points at infinity, namely 
  $\infty_{\pm} = (1:\pm 1:0)$ on the model of $X$ given by the closure of~$y^2=f(x)$ in the weighted
  projective plane $\PP_{1,g+1,1}$.
Since the local heights are bi-additive, it suffices to compute $h_p(P-Q,R-S)$
for distinct points $P,Q,R,S$.
As in~\cite{BBHeights} we assume,  for simplicity and for implementation
reasons (see Remark~\ref{R:extensions})
that these points are all defined over $\Q_p$.

Our method crucially depends on the existence of a nontrivial
degree-0 divisor supported at infinity, namely $D_\infty\colonequals \infty_- - \infty_+$. 
We show in Sections~\ref{sec:Even-two-infinities} and~\ref{sec:affine} that the computation of $h_p(P-Q,R-S)$ can be reduced essentially to the 
computation of Coleman integrals of the differentials 
$$\omega_i\colonequals \frac{x^idx}{y},\quad i=0,\ldots,2g\,. $$
For $P-Q=D_\infty$, this reduction is straightforward, since the residue
divisor of $\omega_g$ is $D_\infty$ (see
Section~\ref{sec:Even-two-infinities}).
Using a change of variables and further reductions to divisors of a specific shape,
we also reduce the computation of the relevant integrals $\int
\omega_{P-Q}$ for most other divisors $P-Q$ to the computation of $\int
\omega_g$ (see Section~\ref{sec:affine}).
Coleman integrals of the differentials $\omega_i$ 
can be computed easily using Balakrishnan's algorithm
\cite{Jen-Even-Degree-CI}, which is implemented in {\tt SageMath}~\cite{Sage}.

No nontrivial degree~0 divisor supported at $\infty$ exists when $f$ has odd degree.
Instead, the authors of~\cite{BBHeights}  compute Coleman
integrals of differentials of the third kind  using a clever but 
more complicated 
approach that involves high precision computations in 
local coordinates of points over high degree extensions of
$\Q_p$.
Our algorithm in fact extends to odd degree, by transforming to an even degree model. 
We give some timings in~\S\ref{subsec:implementation}; these show that our algorithm
is indeed faster than the one from~\cite{BBHeights}.  

A {\tt SageMath}-implementation of our algorithm is discussed
in~\cref{subsec:implementation} and can be obtained 
from~{\url{https://github.com/StevanGajovic/heights_above_p}}. 
We give a precision analysis in Section~\ref{sec:prec-p-adic-heights}.
For our algorithm, we assume that the points $P,Q,R,S$ satisfy a mild condition
(see \sg{Condition 1 in Definition~\ref{def:conds}}, and see~\S\ref{subsec:weakening} on how this 
requirement may be weakened).

\subsection{Application~I: Quadratic Chabauty for rational points}\label{subsec:qcrat}
  The quadratic Cha\-bauty method for rational points was introduced by
  Balakrishnan and Dogra in~\cite{Jen-Netan-QCRP1}. In its simplest form,
  it  can be used to
  compute the rational points on certain curves $X/\Q$ of genus $g>1$
  having Mordell--Weil rank
  $r\colonequals \rk J(\Q) =g$ and Picard number $\rk\mathrm{NS}(J)>1$, 
  where $J$ is the Jacobian variety of $X$. 
The idea is to fix a prime $p$ of good reduction such that $J(\Q)$ has finite index in $J(\Q_p)$
and to use quadratic relations in the image
of the abelian logarithm $\log\colon J(\Q)\otimes \Q_p\to
\mathrm{H}^0(X_{\Q_p},\Omega^1)^\vee$ coming from $p$-adic heights to write down a locally analytic function $\rho\colon X(\Q_p)\to \Q_p$ which has only
finitely many zeroes and vanishes along $X(\Q)$. This is similar to the
method of Chabauty--Coleman, which uses linear relations and requires
$r<g$.
  The approach of Balakrishnan and Dogra fits into Kim's non-abelian
  Chabauty framework  \cite{Kim2005MFG} and \cite{Kim2009Selmer}, which
  aims to generalize Chabauty--Coleman using $p$-adic Hodge theory
  and arithmetic fundamental groups.
  Their results have  been
  generalized~\cite{Jen-Netan-QCRP2}, reinterpreted~\cite{EL19, BMS21} and 
  applied to several explicit examples (see for instance~\cite{QC13, BBBLMTV19, AABCCKW}).

To apply the quadratic Chabauty method in practice, one needs to 
solve for the $p$-adic height pairing
as a bilinear pairing. 
Two methods for this step are discussed in~\cite[\S3.3, \S3.5.4]{BDMTV2}: The first one uses Neko\-v\'{a}\v{r}'s construction and the machinery of~\cite{QC13}, but it requires that $X$ has ``enough'' rational points. 
The second method uses that 
the global $p$-adic height $h$ can be written a linear combination of products of
Coleman integrals of holomorphic forms as in~\cite{BBM1} (see
also~\S\ref{subsec:intro-QC-details} below). To find the coefficients, we need to evaluate these products and the height pairing $h$ in sufficiently many (pairs of) points. For this purpose we can use the techniques of the present article.

For instance, the Atkin--Lehner quotient
$$X_0^+(107)\colon y^2 = x^6 + 2x^5 + 5x^4 + 2x^3 - 2x^2 - 4x - 3$$
has precisely the six rational points $(\pm1, \pm1), \infty_{\pm}$.
In~\cite{BDMTV2}, this is proved using quadratic Chabauty for the prime
$p=61$. Code for this computation can be found
at~{\url{https://github.com/steffenmueller/QCMod}}.
The fairly large prime $p=61$ was chosen for two reasons: First,
it is expected that $\rho$ vanishes in further $p$-adic points
in addition to the rational ones. To show that these 
are not rational, one can use the Mordell--Weil
sieve~\cite{Bruin-Stoll-MW-Sieve}, but for this to have a reasonable chance of
success, the prime $p$ needs to be chosen carefully.
Second, to apply the algorithm from~\cite{BBHeights}, a prime $p$ such that  $X\otimes \Q_p$ has
an odd degree model over $\Q_p$ is required. 
The first condition is also satisfied for $p=7$, but the second one is not. We use the methods of Sections~\ref{sec:Algorithm-computations}, \ref{sec:Even-two-infinities}, and \ref{sec:affine} to compute  $X_0^+(107)(\Q)$ via a combination of the quadratic Chabauty method and the Mordell--Weil sieve as in~\cite{Jen-Netan-QCRP1, QC13, BDMTV2} for $p=7$. 
On 
a single core of a 4-core 2.6 GHz Intel i7-6600U CPU with 8GB RAM, the
entire computation took 47 seconds, in contrast with about~40
minutes for $p=61$ using the \texttt{Magma}~\cite{Magma} implementation of
the algorithm from \cite{BBHeights}. 

Recent work of Duque-Rosero, Hashimoto and Spelier~\cite{DRHS} makes the
geometric quadratic Chabauty method due to Edixhoven and Lido~\cite{EL19}
explicit by reinterpreting it in terms of Coleman--Gross $p$-adic heights. The work described in
the present article can also be used to apply their algorithm in practice.

\subsection{Application~II: Quadratic Chabauty for integral
points}\label{subsec:intro-QC-details}

Prior to the work of Balakrishnan--Dogra, a quadratic Chabauty method to
compute the integral points on hyperelliptic curves  $X$ of odd degree had been developed by
Balakrishnan, Besser and the
second author in~\cite{BBM1}.
This method also requires that $r=g$ and that the closure of
$J(\Q)$ has finite index in $J(\Q_p)$, but has no assumption on
$\rk\mathrm{NS}(J) $.
The main result of~\cite{BBM1} is the existence of a nontrivial locally
analytic function $\rho\colon \calU(\Z_p)\to \Q_p$, where
$\mathcal{U}=\Spec\Z[x,y]/(y^2-f(x))$, and a finite set $T\subset \Q_p$, both computable, such that $\rho$ takes values in $T$ on 
$\calU(\Z)$.
The construction of $\rho$ is fairly technical: it uses Besser's $p$-adic Arakelov theory, an extension of local Coleman--Gross heights to pairs of  
divisors with common support, and double Coleman integrals from tangential base points.
It was further extended in \cite{BBBM-QCNF} to $\O_K$-integral points on certain hyperelliptic curves $X$ of odd degree over number fields $K$. 
The explicit methods based on this approach developed in~\cite{BBM1, BBM2,
BBBM-QCNF} require the computation of local and global $p$-adic heights,
and hence our algorithm leads to a considerable speed-up of these methods. However, these methods are all limited to curves with a global odd degree model.

We develop in~\cite{LinQC} a quadratic Chabauty
method for integral points on even degree hyperelliptic curves over $\Q$
and more general number fields. Using the present article, this can be made explicit,
and we discuss two examples in~\cite{LinQC}, including an example over
$\Q(\sqrt{7})$ that is not amenable to any other existing technique (as far
as we can tell). The idea is to 
decompose the linear functional $\lambda(P)\colonequals
h(\infty_--\infty_+, P)$ into local terms and to write it as a sum of
single Coleman integrals.
A noteworthy feature of this new quadratic Chabauty
method is its simplicity; no double integrals, $p$-adic Arakelov theory,
tangential base points or
$p$-adic Hodge theory are required.

\subsection{Application~III: $p$-adic BSD}\label{subsec:pbsd}
Another
application of our algorithm is the numerical verification of Conjecture~\ref{pbsd}, a $p$-adic Birch and Swinnerton--Dyer
conjecture for modular abelian varieties of $\mathrm{GL}_2$-type with good ordinary reduction, in examples.
So far, this has only been done for elliptic curves and for
Jacobians of hyperelliptic curves that have an odd degree model over $\Q_p$
(see~\cite{BMS16}). Our implementation makes it possible to drop this
condition.
We discuss how to numerically verify Conjecture~\ref{pbsd} and give
a worked example (for $J_0^+(67)$ and $p=11$) in Section~\ref{S:pbsd}.

\subsection*{Acknowledgements}
We thank 
Jennifer Balakrishnan, Alex Best, Francesca Bianchi, Martin L\"udtke and Michael Stoll for many valuable comments and suggestions, and Amnon Besser, Bas Edixhoven, Sachi Hashimoto, Enis Kaya, Kiran Kedlaya, Guido Lido, David Rohrlich, and Jaap Top for helpful discussions. 
We thank the anonymous referees for various particularly helpful remarks and suggestions for
improvement.
We acknowledge support from DFG through DFG-Grant MU 4110/1-1. 
In addition. S.M. was supported by NWO Grant VI.Vidi.192.106, and S.G. was
supported by a guest postdoc fellowship at the Max Planck Institute for
Mathematics in Bonn, by Czech Science Foundation GAČR, grant 21-00420M, and
by Charles University Research Centre program UNCE/SCI/022 during various
stages of this project. 
Part of this research was done during a visit of S.G. to Boston University,
partially supported by a Diamant PhD Travel Grant; S.G. wants to thank
Boston University for their hospitality.
Some of the research in this article forms part of S.G.'s PhD
thesis~\cite{Gajovic-thesis}, written at the University of Groningen under
S.M.'s supervision.

\section{Coleman--Gross heights} \label{sec:Coleman-Gross}
Following~\cite[Section~2]{BBHeights}, we introduce the $p$-adic height
pairing constructed by Coleman and Gross in \cite{Coleman-Gross-Heights}. Let $p$ be a prime number and let $K$ be a number
field. Let $X/K$  be a 
smooth projective geometrically integral  curve of genus $g>0$, with good reduction at all
primes above $p$, and let $J/K$ denote its Jacobian. 
For a non-archimedean place $v$ of $K$ we write $K_v$ for the completion of $K$ at
$v$, $\O_v$ for its ring of integers and $\pi_v$ for a uniformizer, and we denote
$X_v\colonequals X\otimes K_v$.

The $p$-adic height depends on the  
following data, which we fix:
\begin{enumerate}[(a)]
  \item 
A continuous idèle class character
$$\ell\colon \mathbb{A}_K ^{\ast}/K^{\ast}\longrightarrow \mathbb{Q}_p$$
such that the local characters $\ell_\p$ induced by $\ell$, for $\p\mid p$, do not vanish 
on $\O^\ast_{K_\p}$.
\item 
For each $\p\mid p$ a choice of a subspace $W_\p \subset \hdr(X_\p/K_\p)$
    complementary to the space of holomorphic forms.
\end{enumerate}

\begin{remark}\label{rmk:ramified-unramified-over-p}
  The condition in (a) means that the local character $\ell_{\p}$ is ramified,
  which we assume. We require this to define the
  local component of the height using Coleman integration, see Definition
  \ref{def:local-height-above-p}. It is also possible to define the local
  height if the local character $\ell_{\p}$ is unramified, using
  intersection theory exactly as for primes $\q\nmid p$. See also \cite[\S2.2]{BBBM-QCNF}. 
  We use this greater generality in~\cite{LinQC}.
\end{remark}

For any place $\q\nmid p$, we have $\ell_\q(\O_{K_\q}^{\ast})=0$ for continuity reasons, which implies
that $\ell_\q$ is completely determined by $\ell_\q(\pi_\q)$. On the other hand, for
$\p\mid p$, we can decompose 
\begin{equation}\label{trace-on-O_K}
\begin{tikzcd}
 \O_{K_\p}^{\ast}  \arrow{r}{\log_{\p}}  \arrow{rd}{\ell_\p} 
  & K_\p \arrow{d}{t_\p} \\
    & \mathbb{Q}_p
\end{tikzcd}\,,
\end{equation}
where $t_\p$ is a $\mathbb{Q}_p$-linear map. By Condition (a),
it is then
possible to extend $\log_\p$ to 
\begin{equation}\label{trace}
\begin{tikzcd}
 K_\p^{\ast}  \arrow{r}{\log_{\p}}  \arrow{rd}{\ell_\p} 
  & K_\p \arrow{d}{t_\p} \\
    & \mathbb{Q}_p\,.
\end{tikzcd}
\end{equation}
From now on, we fix the branch of the logarithm at $\p\mid p$ to be the one above.

Coleman--Gross define the  height pairing on $J$ by first constructing, for all finite primes $v$
of $K$, local height pairings
$h_v(D_1, D_2)\in\Q_p$ for divisors of degree zero $D_1,D_2\in \Div^0(X_v)$  with disjoint support. These have the property that for $D_1,D_2\in \Div^0(X)$ with disjoint support,
only finitely many $h_v(D_1, D_2)\colonequals h_v(D_1\otimes K_v, D_2\otimes K_v)$ are nonzero.
Hence it makes sense to define 
$$
h(D_1,D_2)\colonequals \sum_{v\sm{\text{ finite}}} 
 h_v(D_1,D_2)\,.
$$
By~\cite[Sections~1,~5]{Coleman-Gross-Heights}, 
the pairing $h$ induces a bilinear pairing
$$
h \colon J(K)\times J(K)\to \Q_p\,.
$$
The local height pairings $h_v$ for $v\nmid p$ are symmetric, whereas the local height pairing $h_{\p}$ for $\p\mid p$ is symmetric if and only if $W_\p$ is isotropic with respect to the cup product
pairing~\eqref{eq:cup-prod} by~\cite[Proposition 5.2]{Coleman-Gross-Heights}.

We  now recall the construction of the local pairings $h_v$.
Since the local terms $h_v$ depend only on $X_v$, we ease  notation as follows:
we write $k=K_v$, $\O = \O_v$, $\pi=\pi_v$, and $C=X_v$.
Let $\chi=\ell_v \colon k^{\ast}\longrightarrow\mathbb{Q}_p$.

\begin{proposition}{\cite[Proposition 1.2]{Coleman-Gross-Heights}}\label{P:htaway}
  If $p\in \O^\ast$, then there exists a unique function $h_v(D_1, D_2)$, defined for
all $D_1,D_2 \in \Div^0(C)$ with disjoint support, which is continuous, symmetric, bi-additive,
and takes values in $\mathbb{Q}_p$, such that for all $f\in k(C)^{\ast}$  we have
  $$h_v(\dv(f),D_2)=\chi(f(D_2))\,,\quad \text{if  }\supp(\dv(f))\cap
  \supp(D_2)=\emptyset\,.$$
\end{proposition} 
We briefly review the construction of $h_v$ in the situation of Proposition~\ref{P:htaway}.
Let $\calC/\O$ be a regular model of $C$ and let $(- \cdot-)$ be the ($\Q$-valued)
intersection multiplicity on $\calC$. Let $\calD_1$ and $\calD_2$ be extensions of
$D_1$ and $D_2$ to $\calC$ such that $(\calD_i\cdot V)=0$
for all vertical divisors $V$ on $\calC$. Then
\begin{equation*}\label{eq:intmult}
  h_v(D_1,D_2)=\chi(\pi)\cdot (\calD_1\cdot \calD_2)\,.
\end{equation*}
Up to a constant, this is equal to the local N\'eron symbol at $v$, discussed in~\cite[Sections~2,~3]{Gross-local-heights}, \sm{which is a local summand of the canonical real-valued height pairing}. It does not depend on the choice of model or on the choices of $\mathcal{D}_1,\mathcal{D}_2$.
\sm{Algorithms to compute this pairing in practice are discussed in  \cite{Hol12,
Mul14, Raymond-David-Steffen}. The {\tt Magma}-command {\tt LocalIntersectionData} uses an implementation of the algorithm in~\cite{Mul14} to compute $h_v$ for all $v\nmid p$; it requires the computation of a regular model using {\tt Magma}'s {\tt RegularModel}-package. A {\tt Magma}-implementation from the algorithm from~\cite{Raymond-David-Steffen}, which does not require $X$ to be hyperelliptic, but currently only works when $K=\Q$, can be obtained from {\url{https://github.com/emresertoz/neron-tate}}.}

The remaining case is when $p\notin\O^\ast$, which we assume from now on
~\footnote{Colmez~\cite{Col98} has shown that there is a uniform
analytic
construction of $h_v$ for all finite primes $v$.}. In particular, $C$ has good reduction.
\begin{definition}\label{diffdef}
A meromorphic differential on $C/k$ is called 
  \begin{itemize}
    \item of the \textit{second kind} if all its residues are~0;
    \item of the \textit{third kind} if it has at most simple poles with residues in $\mathbb{Z}$.
  \end{itemize}
\end{definition}
Let $W=W_v$ be the subspace of $\hdr(C/k)$ chosen above, complementary to
\sm{the image of the holomorphic differentials in $\hdr(C/k)$.}
 
\begin{remark}\label{rmk:W-Unit-Root}
When $C$ has good ordinary reduction, there is a canonical choice of
  the complement $W$. Namely, Frobenius acts on $\hdr(C/k)$ with $2g$
  eigenvalues, of which exactly $g$ are $p$-adic units in the ordinary
  case; the eigenspace corresponding to all of them is called the  {\it
  unit root subspace}, which is automatically  isotropic with respect to
  the cup product \sm{(see for instance~\cite[Theorem~3.1]{Iov00})}.
\end{remark}

Differentials of the form $df$ for $f\in k(X)$ are called {\it exact}; they are of the second kind,
whereas differentials of the form $df/f$ for $f\in k(X)^{\ast}$ are called {\it logarithmic}; they are of the third kind.
We denote by $T(k)$ the group of differentials on $C$ of the third kind and by
$T_l(k)$ the group of logarithmic differentials.
The \textit{residue divisor homomorphism}
$$\Res\colon T(k)\rightarrow \Div^0(C), \hspace{5mm} \Res(\omega)=\sum_{P\in \sg{C(\bar{k})}}\Res_P(\omega)P$$
is surjective.
\sm{We now use the space $W$ to associate a unique differential of the
third kind to a given residue divisor.

Let $\psi\colon T(k)/T_l(k)\rightarrow \hdr(C/k)$ denote 
the canonical homomorphism from
\cite[Proposition~2.5]{Coleman-Gross-Heights}. For our purposes, it
suffices to know that it satisfies Proposition~\ref{prop:Cup-of-psi} below;
we describe a generalization of $\psi$ \sg{in the proof of Proposition~\ref{prop:psi(omega_g)}.}} 
We extend $\psi$ to a linear map on all meromorphic differentials $\omega$ on $C/k$ as follows: Write 
$$\omega = \sum_{i=1}^n
  a_i\mu_i+\eta\,,$$ where $a_i\in \overline{k}$, $\mu_i\in T(k)$ and $\eta$ is a differential
  of the second kind. Then we set
  $$\psi(\omega)\colonequals\sum_{i=1}^n a_i\psi(\mu_i)+[\eta]\in
  \hdr(C/k)\,.$$
  In particular, $\psi$ maps a differential of the second kind to its class
in $\hdr(C/k)$.
\begin{proposition}\cite[Proposition~3.2]{Coleman-Gross-Heights}\label{prop:Unique-differential-3rd-kind}
For any divisor $D$ of degree 0 on $C$, there is a  unique differential form $\omega_D$ of the third
  kind such that
$$\Res(\omega_D)=D, \hspace{5mm} \psi(\omega_D)\in W\,.$$
Hence the choice $W$ induces a section $\Div^0(C)\longrightarrow T(k)$ of the residue
  divisor homomorphism, given by $D\mapsto \omega_D$. If $D=\dv(f)$ is principal, then $\omega_D=\frac{df}{f}$.
\end{proposition} 
Using Proposition~\ref{prop:Unique-differential-3rd-kind}, we can now
define the central object of this article.

\begin{definition}\label{def:local-height-above-p}
Let $D_1$ and $D_2$ be two divisors on $C$ of degree zero with disjoint support. The {\it local $p$-adic height pairing} is given by
$$h_v(D_1,D_2)\colonequals t_v\left(\int_{D_2}\omega_{D_1}\right),$$ 
  where $t_v$ is the linear map introduced in~\eqref{trace-on-O_K} and $\int_{D_2}\omega_{D_1}$ is a Coleman integral. 
\end{definition} 
\sm{The Coleman integral $\int^Q_P\omega$ between two points $P,Q$ such that $\omega$ has no pole at $P$ or $Q$ is given by termwise integration of the convergent power series expansion of $\omega$ around $P$ if $P$ and $Q$ have the same reduction; such integrals are called~\textit{tiny}. 
Coleman~\cite{ColemanI} and Coleman--de Shalit~\cite{Coleman-deShalit} extended this to general $P$ and $Q$ by analytic continuation along Frobenius. Coleman integrals along divisors are then defined by additivity. If $D$ is a divisor of degree~0 and $\omega$ is holomorphic on $C$, then the integral $\int_D\omega$ can also be defined using the abelian logarithm on $J(\Q_p)$.}

The {\it (algebraic) cup product pairing} $\hdr(C/k)\times \hdr(C/k)\longrightarrow k$ is given by 
\begin{equation}\label{eq:cup-prod}
  ([\mu_1],[\mu_2])\mapsto [\mu_1]\cup[\mu_2]\colonequals \sum_{P\in
  C(\bar{k})} \Res_P\left(\mu_2\int\mu_1\right),
\end{equation}
where $\mu_1$ and $\mu_2$ are differentials of the second kind. It
defines a canonical non-degenerate alternating form on $\hdr(C/k)$.

\begin{proposition}{\cite[Proposition 5.2]{Coleman-Gross-Heights}}\label{prop:height-properties}
The local height pairing $h_v(D_1, D_2)$ is continuous, bi-additive and satisfies 
$$ h_v(\dv(f),D_2)=\chi(f(D_2))\,.$$
Furthermore, it is symmetric if and only if the subspace $W$ of $\hdr(C/k)$ is isotropic with respect to the cup product pairing.
\end{proposition} 

Finally, we discuss an expression for the map $\psi$ due to Besser, which
is suitable for explicit computations.

\begin{definition}
For a meromorphic form $\omega$ and a form of the second kind $\rho$ on
  $C$, we define the \textit{local symbol} at a point $A\in C(\bar{k})$ by
$$\langle \omega,\rho\rangle _A\colonequals -\Res_A\left(\omega\int\rho\right),$$
where $\int\rho$ is the function $Q\mapsto\int_{Z}^{Q}\rho$, for some fixed
  point $Z\in C(\bar{k})$.
We define the \textit{global symbol} $\langle \omega,\rho\rangle $ as
$$\langle \omega,\rho\rangle \colonequals \sum_{A}\langle \omega,\rho\rangle _A,$$
where the sum is taken over all points in $C(\bar{k})$ where $\omega$ or $\rho$ have a singularity.
\end{definition} 
We note that although the local symbol depends on
the choice of the point $Z$, the definition of the global symbol does not depend on $Z$. 
In Section~\ref{sec:affine}, Step
(ii), we use the following statement. 
\begin{proposition}{\cite[Proposition 4.10]{Besser-Syntomic-2}}\label{prop:Cup-of-psi}
Let $\omega$ and $\rho$ be as above. Then
  $$\langle \omega,\rho\rangle =\psi(\omega)\cup\sm{ \rho}\,.$$ 
\end{proposition} 
An important application of this result is the following independence result.
\begin{corollary}\label{cor:psi-commutes-with-isomoprhisms}
  Let $\tau\colon C\to C'$ be an isomorphism of curves and let $\omega'$ be a
 differential of the third kind on $C'$. Then we have
  $$\psi(\tau^{*}\omega')=\tau^{*}(\psi(\omega'))\,.$$
\end{corollary}
\begin{proof}
  This follows from Proposition~\ref{prop:Cup-of-psi}
  and~\cite[Lemma~3.8]{Besser-padic-Arakelov}.
\end{proof}

\begin{corollary}\label{cor:height-independence}
  Suppose that $\tau\colon C\to C'$ is an isomorphism of curves. Let $W$
  (respectively $W'$) be a complementary subspace of $\hdr(C/k)$
  (respectively $\hdr(C'/k)$) such that \sm{$\tau^*(W') = W$}.
  Let $h_p$ (respectively $h'_p$) be the local height on $C$ with respect to $W$
  (respectively on $C'$ with respect to $W'$). Suppose that $D_1,D_2\in \Div^0(C)$ have disjoint support. Then we have
$$
h_p'(\tau_*(D_1), \tau_*(D_2)) = h_p(D_1,D_2)\,.
$$
\end{corollary}

\begin{remark}\label{R:other_heights}
  The $p$-adic height should be viewed as a $p$-adic analogue of the 
classical (N\'eron--Tate) canonical height, which
takes values in $\R$. As for the latter, there 
are several constructions of $p$-adic height pairings on abelian varieties taking values
in $\Q_p$: for instance by Mazur and Tate \cite{Mazur-Tate-p-adic-heights}, and by Schneider
\cite{Schneider-p-adic-heights1}.
  Nekov\'{a}\v{r}~\cite{Nek93} has also given a motivic construction of
  $p$-adic heights
  on fairly general Galois representations. When the abelian
variety in question is the Jacobian of a curve with good ordinary reduction at $p$,
  these constructions are all equivalent to the construction of Coleman and
  Gross,
  which we prefer to work with due to its simplicity. 
  \end{remark}

  \subsection{Wide open spaces}\label{subsec:}
\sm{
For a point $P\in C(\Q_p)$, we denote 
the \textit{residue disc} that contains $P$ by 
\begin{equation}\label{res_disc}
D(P)\colonequals\{Q\in C(\Q_p)\,\colon  P\equiv Q\pmod{p}\}\,.
\end{equation}
\textit{Affine discs} are residue discs that do not contain any point at infinity.

  A rigid analytic subspace $\calU$ of the analytification of $C/k$ is called a
  \textit{wide open space} if there are closed discs $D_1,\ldots,D_n$ of radius
  $r_i<1$, each
  contained inside a residue disc, such that $\calU$ is the complement of
  $D_1\cup\ldots\cup D_n$. See~\cite[Section~2]{Besser-Syntomic-2} for more
  information on wide open spaces.

  For a rigid analytic form $\omega$ on $C/k$, we define the
  \textit{residue of $\omega$ along $D_i$} to be the residue along a
  sufficiently small 
  annulus around $D_i$ that is contained in $\calU$ (with respect to a
  suitable local parameter on the annulus). In particular, this
  residue is~0 if $\omega$ 
  has trivial residue at all points in the residue disc containing $D_i$.

  We say that a rigid analytic form on $C/k$
  is \textit{essentially of the second kind} if its residue along each
  $D_i$ is trivial.
  Coleman's integration theory extends to wide open spaces; we refer
  to~\cite[Theorem~4.1]{BBHeights} for its properties. 
  In particular, the
  notion of tiny integrals extends to integrals along annuli in $\calU$
  around a disc $D_i$ if the differential has trivial residue along $D_i$.
}
\section{Precomputations}\label{sec:precomp}
Let \sm{$p>2$ be a prime and let $f\in\Z_p[x]$ be of degree $\ge 3$ without repeated roots, so that $C\colon y^2=f(x)$ is a hyperelliptic curve  over $\Q_p$ with good reduction.
Our algorithm to compute local heights (see Algorithm~\ref{alg:hp} below) assumes that the following data is available:
\begin{itemize}
  \item A basis of $\hdr(C/\Q_p)$. 
  \item The action of Frobenius on $\hdr(C/\Q_p)$.
  \item The cup product matrix on $C$.
  \item A basis for  a subspace $W \subset \hdr(C/\Q_p)$, 
    complementary to the space of holomorphic forms and isotropic with respect to the cup product pairing.
\end{itemize}
Recall that $\omega_0,\ldots,\omega_{g-1}$ form a  basis of
the holomorphic differentials on $C$, where  we denote
$\omega_i\colonequals \frac{x^idx}{y}$ for each $i\ge 0$. 
If $\deg(f)$ is odd, then the classes $[\omega_0],\ldots,[\omega_{2g-1}]$ form a basis of 
$\hdr(C/\Q_p)$. We can then compute the action of Frobenius on this basis using
Kedlaya's algorithm \cite{Kedlaya-MW-reduction}, and the cup product matrix can be computed using local integrals, see~\cite[\S5.1]{BBHeights}. An algorithm to compute a basis for the unit root subspace $W$ (when $p$ is ordinary) is discussed in~\cite[\S6.1]{BBHeights}.

We now discuss how to compute these objects when $\deg(f)=2g+2$ is even, which we shall assume for the remainder of this section. We also assume that $f$ is monic, which suffices for the situation discussed in Section~\ref{sec:Algorithm-computations} below.
In particular, the points at infinity, denoted
$\infty_+$, $\infty_-$, are in $C(\Q_p)$. }



\subsection{Computing a basis of $\hdr(C/\Q_p)$}\label{basis}
\sm{We extend the classes of the differentials $\omega_0,\ldots,\omega_{g-1}$ to a basis of $\hdr(C/\Q_p)$ in
Lemma \ref{lm:dR-from-MW} below;} first we briefly discuss the cohomological background.
Denote the Weierstrass points on the reduction $\overline{C}$ of $C$ modulo $p$ by $W_1,\ldots,W_{2g+2}\in \overline{C}(\overline{\F_p})$. 
Consider the set
\[
V\colonequals \{W_1,\ldots,W_{2g+2},\infty_-,\infty_+\}\subset
\overline{C}(\overline{\F}_p)\,.
\]

Let $U=\Spec\F_p[x,y,1/y]/(y^2-f(x))$ be the affine curve
$\overline{C}\setminus V$. 
\sm{For each $P\in V$, we choose a disc $D_P$ of sufficiently large  radius contained in the residue disc reducing to $P$. Similar
to~\cite{BBHeights}, we consider a wide open space $\calU$ obtained
by removing the discs $D_P$ for all $P\in V$.
Then $U$ is the reduction of $\calU$.
By~\cite[Proposition 4.8]{Besser-Syntomic-2}, there is an exact sequence:
}


\begin{equation}\label{eq:Besser-Exact-Sequence}
\begin{tikzcd}
    \displaystyle  0 \arrow[r] & \hdr(C/\Q_p) \arrow[r] &
  \hdr(\calU) \arrow[r, "\Res"] & \bigoplus_{\substack{P\in
  V}}\C_p\,,
\end{tikzcd}
\end{equation} 
\sm{where for $P\in V$, the $P$-component of the map $\Res$ is the residue
  along $D_P$.

The work of Baldassarri and Chiarellotto~\cite{Baldassarri-Chiarellotto}
implies that we have
$\hdr(C/\Q_p)\simeq \hrig(\overline{C})$, where
$\hrig(\overline{C})$ is the first rigid cohomology of $\overline{C}$.
Similarly,~\cite{Baldassarri-Chiarellotto} also shows that
$\hrig(U)\simeq \hdr(\calU)$. In rigid cohomology, the exact sequence
becomes } 
\begin{equation}\label{eq:Cohomology-Exact-Sequence}
\begin{tikzcd}
\displaystyle  0 \arrow[r] & \hrig(\overline{C}) \arrow[r] &
  \hrig(U) \arrow[r, "\Res"] & \bigoplus_{\substack{P\in
  V}}\C_p\,,
\end{tikzcd}
\end{equation} 
\sm{
where for $P\in V$, the $P$-component of $\Res$ is the residue $\Res_P$ at
$P$ (see~\cite[Theorem 3.11]{Tuitman-P1-two}).}

By
  \cite[Proposition 1.10]{Berthelot-isomorphism-rigid-MW}, $\hrig(U)$ is isomorphic to  the first Monsky-Washnitzer
    cohomology group $\hmw(U)$ of $U$. \sm{For the construction of the latter see~\cite{Marius-MW-Cohomology}; it has the advantage of
    being amenable to computations.
Let $\hmw(U)^{\pm}$ be the $\pm$ component of $\hmw(U)$ with respect to the hyperelliptic involution. 
By~\cite[\S3.2]{Harrison-Even-Deg-MW}, $\{[\omega_0], \ldots,[\omega_{2g}]\}$
is a basis of $\hmw(U)^-$ and $\{ [\mu_0],\ldots, [\mu_{2g+1}]\}$ is a basis of $\hmw(U)^+$ 
where $\mu_i \colonequals \frac{x^idx}{2y^2}$.
Hence we have}
\begin{equation}\label{eq:H1MW-even-basis}
\hrig(U)\simeq\Span([\omega_0],\ldots,[\omega_{2g}])\oplus\Span([\mu_0],\ldots,[\mu_{2g+1}])\,,
\end{equation}
    \sm{where $\Span$ denotes the $\Q_p$-linear span.}


Since $\omega_i$ is holomorphic on $C(\overline{\Q}_p)\setminus \{\infty_-,\infty_+\}$ for all $i\geq 0$, and holomorphic on $C$ for $0\leq i \leq g-1$, we have
\begin{itemize}
\item $\Res(\omega_i) = (0,\ldots,0,0,0)$, for $0\leq i \leq g-1$;
\item $\Res(\omega_g) =  (0,\ldots,0,1,-1)$, as we will see in Proposition \ref{prop:residue-of-wg};
\item $\Res(\omega_i) =  (0,\ldots,0,c_i,-c_i)$ for each $g+1\leq i\leq 2g$ and some  $c_i\in \Q_p$.
\end{itemize}

On the other hand, we compute that for $1\leq j\leq 2g+2$, $0\leq i\leq
2g+1$, we have
\[
\Res_{(a_j,0)}(\mu_i)=\Res_{(a_j,0)}\left(\dfrac{x^i}{f'(x)}\dfrac{dy}{y}
\right)= \dfrac{a_j^i}{f'(a_j)}\,,
\]
\sm{where the $a_j\in \overline{\Q_p}$ are the roots of $f$ and $(a_j,0)$ reduces to $W_j$.}

\begin{lemma}\label{lm:positive-cohomology-misses-kernel}
We have $\ker(\Res)\cap \Span(\mu_0,\ldots,\mu_{2g+1}) = \{0\}$.
\end{lemma}
\begin{proof}
Assume that there are $\gamma_0,\ldots,\gamma_{2g+1}\in \Q_p$ such that
  \[\Res(\gamma_0\mu_0+\cdots+\gamma_{2g+1}\mu_{2g+1})=(0,\ldots,0,0,0)\,.\]
  It follows that for any $1\leq j\leq 2g+2$ we have
\[
\sum_{i=0}^{2g+1}\gamma_i \dfrac{a_j^i}{f'(a_j)}=0 \implies
  \sum_{i=0}^{2g+1}\gamma_i a_j^i=0\,.
\]
In other words, the polynomial $\sum_{i=0}^{2g+1}\gamma_i x^i$ has at least $2g+2$ distinct zeros $a_1,\ldots,a_{2g+2}$, hence, it is the constant zero polynomial, so $\gamma_0=\cdots=\gamma_{2g+1}=0$.
\end{proof}

Lemma~\ref{lm:positive-cohomology-misses-kernel} and the discussion
preceding it imply: 
\begin{corollary}\label{cor:exact-sequence-dr-cohomology-computation}
There is an exact sequence
\begin{equation*}\label{eq:cohomology-dR-MWminus-exact-sequence}
\begin{tikzcd}
\displaystyle    0 \arrow[r] & \hdr(C/\Q_p) \arrow[r] &
  \Span(\omega_0,\ldots,\omega_{2g}) \arrow[r, "\Res"] &
  \mathbb{C}_p\oplus \mathbb{C}_p\,,
\end{tikzcd}
\end{equation*} 
where $\Res(\omega)=(\Res_{\infty_-}\omega, \Res_{\infty_+}\omega)$.
\end{corollary}

We now use Corollary \ref{cor:exact-sequence-dr-cohomology-computation}  to
construct an explicit basis of $\hdr(C/\Q_p)$. For $0\leq i \leq g-1$, let
$\eta_i\colonequals \omega_i$ and for $g\leq i \leq 2g-1$, let $\eta_i
\colonequals\omega_{i+1} -
c_i\omega_g$, with $c_i$ as in (iii). 

\begin{lemma}\label{lm:dR-from-MW}
The classes of the differentials $\eta_0,\ldots,\eta_{2g-1}$ form a basis of $\hdr(C/\Q_p)$.
\end{lemma}
\begin{proof}
From the above, it follows that $\Span(\eta_0,\ldots,\eta_{2g-1})\subset \ker(\Res)$. Since the classes of the $\eta_i$ are independent in $\hdr(C/\Q_p)$ by \eqref{eq:H1MW-even-basis}, the claim follows by comparing dimensions. 
\end{proof}

\subsection{The action of Frobenius on $\hdr(C/\Q_p)$}\label{frob}
An algorithm for the computation of the matrix $\Phi=(f_{i,j})\in
\Q_p^{(2g+1)\times(2g+1)}$ of Frobenius acting on $\Span([\omega_0], \ldots,[\omega_{2g}])
\subset \hmw(U)^-$
is given in~\cite{Harrison-Even-Deg-MW}. \sm{More precisely, this is a
lift of Frobenius to a rigid analytic morphism that maps the wide open space
$\calU$ to a space with the same cohomology.}
Hence the matrix of Frobenius on $\hdr(C/\Q_p)$ with respect to the basis $[\eta_0],\ldots,[\eta_{2g-1}]$, which we denote by $\Frob$, is 
$$\Frob=\begin{pmatrix} 
f_{0,0} & \ldots & f_{0,g-1} &  f_{0,g+1}-c_{g+1}f_{0,g} & \ldots & f_{0,2g}-c_{2g}f_{0,g} \\
\vdots & \ddots & \vdots & \vdots & \ddots & \vdots\\
f_{g-1,0} & \ldots & f_{g-1,g-1}  & f_{g-1,g+1}-c_{g+1}f_{g-1,g} & \ldots & f_{g-1,2g}-c_{2g}f_{g-1,g}\\
f_{g+1,0} & \ldots & f_{g+1,g-1}  & f_{g+1,g+1}-c_{g+1}f_{g+1,g} & \ldots & f_{g+1,2g}-c_{2g}f_{g+1,g}\\
\vdots & \ddots & \vdots & \vdots  & \ddots & \vdots\\
f_{2g,0} & \ldots & f_{2g,g-1} & f_{2g,g+1}-c_{g+1}f_{2g,g} & \ldots & f_{2g,2g}-c_{2g}f_{2g,g}
\end{pmatrix}.$$

\subsection{The cup product matrix}\label{cup} 
Since all $\eta_i$ are holomorphic on $U$, 
the cup product matrix on $C$, which we denote by $M=(m_{ij})_{i,j}$, is given by 
\[
m_{ij}=\Res_{\infty_+}\left(\eta_j\int\eta_i\right)+\Res_{\infty_-}\left(\eta_j\int\eta_i\right) =2\Res_{\infty_+}\left(\eta_j\int\eta_i\right)\,.
\]
The final equality holds because if $(s_x,s_y)$ are local coordinates of one point at infinity, for example induced by $s_x=\frac{1}{x}$, then $(s_x,-s_y)$ are local coordinates of the other one, and $\eta_j\int\eta_i$ is an even function in $y$. 

\subsection{Isotropic complementary subspaces}\label{subsec:subspace}
\sm{Recall that part of the data required to construct the local height $h_p$ is a 
choice of a subspace $W \subset \hdr(C/\Q_p)$ 
    complementary to the space of holomorphic forms. In applications, we
    often find it convenient to work with symmetric height pairings, so in view of Proposition~\ref{prop:height-properties}, we want $W$ to be isotropic with respect to the cup product pairing. Suppose we have computed the $\eta_i$, the matrix $\Phi$ and the cup product matrix $M$.}

One possible choice of isotropic complementary subspace $W$ is a subspace 
whose basis, together with $[\eta_0],\ldots,[\eta_{g-1}]$,
forms a symplectic basis. \sm{Such a basis can be computed easily from $\eta_0,\ldots,\eta_{g-1}$ and the matrix $M$ using linear algebra.}

Alternatively, if $p$ is ordinary, we can compute the unit root
subspace $W$ \sm{(see Remark~\ref{rmk:W-Unit-Root})} to any desired precision using the
following result. We denote by $\phi$ the lift of Frobenius to (a certain wide open
subspace of) $C$ constructed by
Kedlaya~\cite{Kedlaya-MW-reduction} and
Harrison~\cite{Harrison-Even-Deg-MW}, and we denote by $\phi^*$ the action of Frobenius 
on functions and differentials. 
\begin{proposition}\label{prop:unit-root-subspace}{\cite[Proposition 6.1]{BBHeights}}
\sg{When $p$ is a prime of good ordinary reduction, the elements } ${\phi^*}^n(\eta_g), \ldots , {\phi^*}^n(\eta_{2g-1})$
form a basis of the unit root subspace modulo $p^n$.
\end{proposition}

\begin{remark}\label{unit-root}
For some applications, for instance, to quadratic Chabauty, we can take $W$ to be any
  complementary subspace isotropic with respect to the cup product.
 In the good ordinary case, the
  global Coleman--Gross height with respect to the unit root subspace is of
  special importance due to its appearance in the 
  $p$-adic Birch and Swinnerton--Dyer conjecture formulated in~\cite{BMS16}.
  Hence we need to
  compute $p$-adic heights with respect to the unit
  root subspace to gather empirical evidence for this conjecture. See
  Section~\ref{S:pbsd}.
\end{remark}
\section{Computing local heights above $p$ on even degree hyperelliptic curves}\label{sec:Algorithm-computations}

As in the previous section, we let 
$p>2$ be prime and we let $C\colon y^2=f(x)$\sg{, where $f\in\Z_p[x]$} be a hyperelliptic curve with good reduction over $\Q_p$. In \cite{BBHeights},  Balakrishnan and Besser
introduce an algorithm to compute the local height $h_p$ on $C$ with respect to
the unit root subspace when $\deg(f)$ is odd. Here we partially follow
their strategy, but modify the key steps to significantly simplify and
speed up the algorithm 
when $\deg(f)$ is even. 

\sg{Suppose that there exists a point $P_0\in C(\Q_p)$ such that $\ord_p(y(P_0))=0$. In this case, we can apply a transformation (see Equation~\eqref{eq:change-of-variables} in \sg{Section~\ref{sec:affine}}) that does not
change the local height by Corollary~\ref{cor:height-independence} and
results in a model $C\colon y^2=f(x)$ where $f\in\Z_p[x]$ is monic of even degree. }
We will assume that we are in this situation from now on.
\sm{
As in the previous section, we fix a complementary subspace
$W\subset \hdr(C/\Q_p)$, isotropic with respect to the cup product pairing,
and a continuous homomorphism $\chi\colon
\Q_p^{\ast}\longrightarrow\mathbb{Q}_p$, obtained from a continuous idèle class character.
In this situation, $\chi=\log_p$ is the Iwasawa branch of the logarithm, determined by $\log_p(p)=0$, and the linear map $t_\p$ in~\eqref{trace-on-O_K}  is a scalar multiple of the identity map. In the following, we assume for ease of notation that $t_\p$ is the identity map, and we call $\chi$ the (local component of) the \textit{cyclotomic id\`ele class character}.}


\subsection{Reduction using bi-additivity}\label{subsec:reduction-bi-additivity}

Any divisors $D_1, D_2\in \Div^0(C)$ with disjoint support can be written as 
$$D_1=\sum_{i=1}^n (P_i-Q_i), \hspace{5mm} D_2=\sum_{j=1}^m (R_j-S_j)\,$$
for pairwise distinct points $P_i,Q_i,R_j,S_j\in C(\overline{\Q}_p)$. For
simplicity (and in our implementation), we will assume
that all these points are already defined over $\Q_p$, but the results below also hold in greater generality.
Using the bi-additivity of the local height pairing, we can reduce the computation of
$h_p(D_1,D_2)$ (with respect to $W$ and $\chi$) to the computation of $h_p(P_i-Q_i,R_j-S_j)$ for all $i,j$. Hence, to
compute $p$-adic heights, it suffices to compute $h_p(P-Q,R-S)$ for distinct points $P,Q,R,S\in C(\Q_p)$. 
\sm{In Algorithm~\ref{alg:hp}, we introduce a method to compute
$h_p(P-Q,R-S)$ that essentially follows the strategy
applied by Balakrishnan and Besser for odd degree.}


\subsection{Conditions on points}\label{subsec:conditions-points} 


Let $\iota\colon C\rightarrow C$ denote the hyperelliptic involution. 
\sm{
Recall that for $P\in C(\Q_p)$ we denote by $D(P)$ the residue disc containing it.
\begin{definition}\label{def:conds} We say that 
   distinct points $P_1,P_2,P_3,P_4\in C(\Q_p)$ satisfy \textit{Condition~1} if
  $$\{D(P_1), D(P_2), D(\iota(P_1)), D(\iota(P_2))\}\cap
  \{D(P_3),D(P_4)\} = \emptyset.$$
  We say that they satisfy \textit{Condition~1'} if 
  $$\{\iota(P_1),\iota(P_2)\}\cap \{P_3,P_4\} = \emptyset.$$ 
\end{definition}

Note that Condition~1 is strictly stronger than Condition~1'.
Our strategy for computing $h_p(P-Q, R-S)$, which imitates the strategy from \cite{BBHeights}, relies on the decomposition of divisors of degree zero into the symmetric and antisymmetric part with respect to $\iota$, see for instance Lemma \ref{L:symm-antisymm}. 
This strategy requires that Condition~1' is satisfied for the points $P,Q,R,S$.


In the following, we assume that Condition~1 is satisfied for the points $P,Q,R,S$ for practical purposes; concretely, for the computation of Step~\eqref{ints} of Algorithm~\ref{alg:hp}. If Condition~1 were not satisfied, then at least one of the endpoints of the resulting Coleman integral in Section~\ref{sec:Even-two-infinities} would belong to a residue disc at infinity. Since
Balakrishnan's implementation \cite{Jen-Even-Degree-CI} of even degree Coleman integration in {\tt SageMath} assumes that this is not the case, we would not be able to compute this integral. Similarly, in Section~\ref{sec:affine}, we use a change of variables to compute a certain Coleman integral, which would lead to the same issue. We explain in \cref{subsec:weakening} how to circumvent this issue in theory, thus weakening Condition~1 to Condition~1', but we have not attempted to implement this.}


In fact, Balakrishnan's implementation of the Coleman integral 
   $\int_{R}^{S} \omega$ assumes that
    $\omega$ is a differential that has no poles at $R$ and $S$ and
    $D(R)= D(S)$ or the following two conditions are satisfied:
\begin{itemize}
\item[(a)] neither $D(R)$ nor $D(S)$  is a disc at infinity and
\item[(b)] if $D(R)$  is a  Weierstrass disc, then $R$ is a  Weierstrass point
  (and the same for $D(S)$).
\end{itemize} 
   So $R$ and $S$ need to satisfy the following Condition~2 if we want to carry out Step~\eqref{ints} of Algorithm \ref{alg:hp} directly. 

\sg{

\begin{definition}
  We say that two distinct affine points $P_1, P_2\in C(\Q_p)$ satisfy \textit{Condition~2} if they are in the same residue disc or if \[\ord_p(y(P_1)),\ord_p(y(P_2))\in\{0,\infty\}.\]
\end{definition}

In fact, Step~\eqref{ints} of Algorithm \ref{alg:hp} assumes that $P$ and $Q$ also satisfy Condition~2; see Remark~\ref{stepivcond2}.
However, it turns out that it is possible to remove Condition~2 while keeping Condition~1. We discuss how to do this in~\S\ref{subsec:weierstrass}.
Therefore, our implementation in {\tt SageMath} does not assume Condition~2. Nevertheless, since we believe that it simplifies the exposition, we will assume it for both $P,Q$ and for $R,S$.

For convenience, we summarize our running assumptions here:
\begin{ass}\label{assumptions}
\hfill
\begin{enumerate}
\item $P,Q,R,S$ satisfy Condition~1.
\item $P,Q$ satisfy Condition~2.
\item $R,S$ satisfy Condition~2.
\end{enumerate}
\end{ass}
}


\subsection{The algorithm}
\sm{We now present our algorithm to compute the local height $h_p$. We first state it, and then we discuss the various steps in Sections~\ref{sec:Even-two-infinities} and~\ref{sec:affine}.}

\begin{alg} \label{alg:hp}  {\bf (Computation of the local height)}
\hfill

\sm{ {\bf Input:} 
\begin{itemize}
    \item A hyperelliptic curve $C\colon y^2=f(x)$ over $\Q_p$ of good reduction. 
    \item Four distinct points $P,Q,R,S\in C(\Q_p)$ satisfying Assumption~\ref{assumptions}.
    \item Differentials $\eta_i$ as in~\S\ref{basis}.
    \item The matrix of Frobenius $\Phi$ as in~\S\ref{frob}.
    \item The cup product matrix $M$ as in~\S\ref{cup}.
    \item A basis $\kappa_0,\ldots,\kappa_{g-1}$ for an isotropic complementary subspace $W$ as in~\S\ref{subsec:subspace}.
\end{itemize}
{\bf Output:} The local height $h_p(P-Q,R-S)$ with respect to $W$ and the local component of the cyclotomic id\`ele class character.}
\begin{enumerate}[(i)]
  \item\label{om'} Find some differential $\omega'$ such that $\Res(\omega')=P-Q$.
  \item\label{psiom'} Compute $\psi(\omega')$.
\item\label{coeffs}  Compute the unique coefficients $u_0,\ldots,u_{g-1}\in
  \Q_p$ 
    such that
   \sg{$\omega\colonequals \omega' -\sum^{g-1}_{i=0}u_i\eta_i$} satisfies $\Res(\omega)=P-Q$ and $\psi(\omega)\in W$.
\item\label{ints}  Compute the Coleman integrals $\int_S ^R \omega'$ and \sg{$\int_S ^R \eta_i$} for
  $i=0,\ldots,g-1$ and return $h_p(P-Q,R-S) =\int^R_S\omega$. 
\end{enumerate}
\end{alg}

\sg{{\it Step \eqref{coeffs}} only requires linear algebra, and it is the same for any distinct points $P,Q,R,S\in C(\Q_p)$ satisfying Assumption~\ref{assumptions}.
\begin{lemma}\label{lm:represent-psi-in-mixed-basis}
Let $\omega'$ be a differential of the third kind such that $\Res(\omega')=P-Q$. Using a base change formula, we express
\[
\psi(\omega')=u_0\eta_0+\cdots+u_{g-1}\eta_{g-1}+u_g\kappa_0+\cdots + u_{2g-1}\kappa_{g-1}
\]
and set
\[
\omega\colonequals\omega'-(u_0\eta_0+\cdots+u_{g-1}\eta_{g-1})\,.
\]
 Then $\Res(\omega) = P-Q$ and $\psi(\omega)\in W$.
\end{lemma}}

We will distinguish two cases: first, we assume
in Section~\ref{sec:Even-two-infinities} that 
one of the divisors is $\infty_- - \infty_+$. This case has no analogue in
odd degree, and we found no simple way to adapt the strategy employed
in~\cite{BBHeights} to this case. Instead, we develop a new approach that
turns out to be faster; the main task is to compute the matrix of Frobenius
on Monsky--Washnitzer cohomology,
which we need anyway for Coleman integration. In fact, 
this case is particularly interesting for two reasons. Namely,
the heights in our new quadratic Chabauty
algorithm for integral points (see~\cite{LinQC}) are precisely of this
form. Moreover, we can reduce the case where all four points are 
affine to it; see Section~\ref{sec:affine}, especially Remark~\ref{R:aff_inf}.

We will not consider the case when one divisor contains one point at infinity
in its support and the other divisor contains the other point in its support, see 
\sg{Assumption~\ref{assumptions}} . When there is only one point at infinity among
these points, we can reduce to the cases already considered,
see~\S\ref{subsec:inf}.

\sg{\section{One divisor supported at $\infty$}\label{sec:Even-two-infinities}}

\sg{Suppose that Assumption~\ref{assumptions} holds for $R,S, \infty_-,\infty_+$.}
\footnote{We explain in \S\ref{subsec:weierstrass} how to remove the assumption that $R$ or $S$ belongs to a Weierstrass disc.}
We now explain how to compute $h_p(\infty_- - \infty_+,R-S)$ using Steps \eqref{om'}--\eqref{ints}.

{\it Step \eqref{om'}} is solved by the following result.
\begin{proposition}\label{prop:residue-of-wg}
Let $\omega'=\omega_g$. Then $\Res(\omega')=\infty_- -\infty_+$.
\end{proposition}

\begin{proof}
  Recall that $\omega_g$ is holomorphic on $C(\overline{\Q}_p)\setminus \{\infty_-,\infty_+\}$. At $\infty_{\pm}$, we can take the uniformizer $t=\frac{1}{x}$, and these two points are distinguished by the function $\frac{x^{g+1}}{y}$, since
\[
\dfrac{x^{g+1}}{y}\left(\infty_+\right)=1, \hspace{2mm} \dfrac{x^{g+1}}{y}\left(\infty_-\right)=-1.
\]
Differentiating the relation $tx=1$ gives $\dfrac{dx}{x}=-\dfrac{dt}{t}$. Upon rewriting
$$\dfrac{x^gdx}{y}=\dfrac{x^{g+1}}{y}\dfrac{dx}{x}=-\dfrac{x^{g+1}}{y}\dfrac{dt}{t}\,,$$
we see that 
 \sg{$\Res_{\infty_+}\omega_g=-1$ and $\Res_{\infty_-}\omega_g=1$}.
\end{proof}

{\it Step \eqref{psiom'}.}
In the spirit of \cite[Algorithm 4.8]{BBHeights}, we define the differential  $$\alpha=\,\sg{\phi^*(\omega_g)-p\omega_g\,}.$$
The strategy employed in~\cite{BBHeights} is to first compute $\psi(\omega') $ and then
deduce $\psi(\alpha)$ from this. Here, we do the opposite, using the following result.

\begin{proposition}\label{P:alphasecond}\hfill
\begin{itemize}
\item[(a)] The differential $\alpha$ is holomorphic at both $\infty_{\pm}$.
\item[(b)] \sg{The
  differential $\alpha$ is essentially of the second kind. More precisely,
    $\alpha$ is holomorphic at non-Weierstrass points, and has  residue~0
    along the discs $D_P$ for all Weierstrass points $P$.} 
\end{itemize} 
\end{proposition}

\begin{proof}
To prove (a) we only consider one point at infinity, say $\infty_-$; the
  proof is the same for the other one. The action of Frobenius on the
  uniformizer $t=\frac{1}{x}$ is 
\[
\phi^* t=\phi^*\left(\dfrac{1}{x}\right)=\dfrac{1}{x^p}=t^p\,.
\]
We know that $\dfrac{x^gdx}{y}=\dfrac{dt}{t}+A(t) dt $, where $A(t)\in \Q_p[[t]]$.
Hence
\[
\phi^*\left(\dfrac{x^gdx}{y}\right)=p\dfrac{dt}{t}+pt^{t-1}A(t^p)dt\,.
\]
It follows that
\[
\alpha=(pt^{p-1}A(t^p)-pA(t))dt
\]
is holomorphic at $\infty_-$. 

For (b), we first show that $\alpha$ is holomorphic at non-Weierstrass
  points. We already know that \sg{$\omega_g$} is holomorphic at every point
  $P=(a,b)\in C(\overline{\Q}_p)$ with $b\neq 0$. The function $t=x-a$ is a uniformizer at $P$, thus $dx=dt$. We compute 
\begin{equation}\label{eq:Frob-of-wg}
 \sg{\phi^*(\omega_g)}=\phi^*\left(\dfrac{x^gdx}{y}\right)=\dfrac{px^{pg+p-1}dx}{y^p}\displaystyle\sum_{i\geq
  0}\dbinom{-\frac{1}{2}}{k}\dfrac{(f(x^p)-f(x)^p)^i}{y^{2pi}}\,,    
\end{equation}
so that  \sg{$\ord_P\frac{\phi^*(\omega_g)}{dt}=\ord_P\frac{\phi^*(\omega_g)}{dx}\geq 0$}.
Hence the $t$-adic valuation of $\alpha$ is nonnegative.

A uniformizer at a Weierstrass point $P=(a,0)$ is given by $t=y$. Then $x$ is an
  even function of $t$ as we now explain. Let $x=\sum_{n\geq 0} e_nt^n$ for
  some $e_n\in\overline{\Q}_p$, then, inductively equating powers of $t$ on
  both sides of the equation $t^2=f(x)$, we see that $f(e_0)=0$, and then for odd $n$, $f'(e_0)e_n=0$. Since $f'(e_0)\neq 0$, it follows that $e_n=0$ for odd $n$.
We now prove that the same holds for the expansion of $\alpha$ in $t$. Using the relation
$$\dfrac{dx}{y}=\dfrac{2dy}{f'(x)}$$
it is clear that 
$$p\dfrac{x^gdx}{y}=p\dfrac{2x^gdy}{f'(x)}$$
has only even powers of $t$. For the other term of $\alpha$, we
  re-express~\eqref{eq:Frob-of-wg} using $dy$:
$$\phi^*\left(\dfrac{x^gdx}{y}\right)=\dfrac{2px^{pg+p-1}dy}{f'(x)y^{p-1}}\displaystyle\sum_{i\geq 0}\dbinom{-\frac{1}{2}}{k}\dfrac{(f(x^p)-f(x)^p)^i}{y^{2pi}},$$
from which we see that all powers of $y$ are even. Hence there are no odd
  powers of $t$ in the expansion of $\alpha$ and therefore \sm{$\alpha$ has
  trivial residue along $D_P$. }
\end{proof}

Recall that \sg{$\phi^*\omega_g=\sum_{i=0}^{2g}f_{g,i}\omega_i$}. 
\begin{proposition}\label{prop:psi(omega_g)}
\sm{
  We have
\begin{equation*}\label{Psi-Omega-From-Alpha}
\psi(\omega_g)=(\Frob-pI)^{-1}\cdot \begin{pmatrix}
f_{g,0} & \cdots & f_{g,g-1} & f_{g,g+1} \cdots & f_{g,2g}
  \end{pmatrix}^t
\end{equation*}}
with respect to the basis $[\eta_0],\ldots, [\eta_{2g-1}]$ of $\hdr(C)$.
\end{proposition}
\begin{proof}
Harrison's extension \cite{Harrison-Even-Deg-MW} of Kedlaya's algorithm
  \cite{Kedlaya-MW-reduction} shows that there is a certain overconvergent
  function $v$ on $\calU$
  such that we have 
\[
\alpha=\sg{\phi^*(\omega_g)-p\omega_g}=\sum_{0\leq i\leq 2g, i\neq g}f_{g,i}\omega_i + (f_{g,g}-p)\omega_g + dv\,.
\]
Recall that $\eta_i=\omega_i$ for $0\leq i<g$ and $\eta_i
=\omega_{i+1} -
c_i\omega_g$ for $g\leq i \leq 2g-1$ and some $c_i\in\Q_p$, so we can further write
\[
\alpha=\sum_{0\leq i< 2g}f_{g,i}\eta_i+\nu\omega_g+dv\,,
\]
for some $\nu\in\Q_p$. Since $\alpha$ is holomorphic at $\infty_-$  by Proposition~\ref{P:alphasecond}(a), $\eta_i$ are holomorphic, and $dv$ has all residues zero, we obtain that $\nu=0$ and
$\alpha=\sum_{0\leq i< 2g}f_{g,i}\eta_i+dv$. 
  Thus, we have that
\begin{equation}\label{alphahdrU}
 \sg{[\alpha]=\sum_{0\leq i< 2g}f_{g,i}[\eta_i]\in \hdr(\calU). }
\end{equation}

\sg{We now use an extension of the map $\psi$ to rigid analytic forms due to
Besser.
By~\cite[Proposition~4.8]{Besser-Syntomic-2}, the
sequence~\eqref{eq:Besser-Exact-Sequence} is split exact, so we obtain
a canonical projection ${\bf p}\colon \hdr(\calU)\to \hdr(C)$ which is the
identity on the image of $\hdr(C)$. For a rigid analytic form $\omega$
that is holomorphic on $\calU$, we define $\psi(\omega)\colonequals
\bf{p}([\omega])$. This extends the map $\psi$, see~\cite[Remark~4.12]{Besser-Syntomic-2}.

Recall that $\alpha$ is holomorphic along $\calU$
and essentially of the second kind on $C$ by Proposition~\ref{P:alphasecond}. Then the class of $\alpha$ in $\hdr(\calU)$ lies in the image of $\hdr(C)$ by
exactness of~\eqref{eq:Besser-Exact-Sequence}. Hence~\eqref{alphahdrU} immediately implies that this class is the image of $\sum_{0\leq i< 2g}f_{g,i}[\eta_i]\in \hdr(C)$. By the extended definition of $\psi$, we have that
$$\psi(\alpha)=\begin{pmatrix}
f_{g,0} & \cdots & f_{g,g-1} & f_{g,g+1} \cdots & f_{g,2g}
\end{pmatrix}^t.$$ 
with respect to the basis $[\eta_0],\ldots, [\eta_{2g-1}]$ of $\hdr(C)$.
The construction of the splitting
in~\cite[Proposition~4.8]{Besser-Syntomic-2} implies that the map $\psi$ satisfies
\begin{equation}\label{phipsi}
  \psi\circ \phi^*={\phi^*}\circ \psi\,.
\end{equation}
Therefore, $\psi(\alpha)=(\Frob - pI)\psi(\omega_g)$, which finishes the proof.
}

\end{proof}

\sg{Using Lemma~\ref{lm:represent-psi-in-mixed-basis}, we compute $u_0,\ldots,u_{g-1}\in \Q_p$ such that for 
$\omega=\omega_g-(u_0\eta_0+\cdots+u_{g-1}\eta_{g-1})$, we have $\Res(\omega) = \infty_- -\infty_+$ and $\psi(\omega)\in W$.}

\begin{corollary}\label{hpinfformula}
The local  height pairing $h_p$ with respect to $\ell_p$ and $W$ satisfies 
\[
h_p(\infty_- - \infty_+,R-S) = \int_S^R \omega_g -
  (u_0\int_S^R\eta_0+\cdots+u_{g-1}\int_S^R\eta_{g-1})\,,
\]
where $u_0,\ldots,u_{g-1}$ are defined in Lemma \ref{lm:represent-psi-in-mixed-basis}.
\end{corollary}

{\it Step \eqref{ints}.} 
Using \cite{Jen-Even-Degree-CI}, we compute 
\[
\int_S^R \omega_g,\;\; u_0\int_S^R\eta_0+\cdots+u_{g-1}\int_S^R\eta_{g-1}\,.
\]
\sm{By Corollary~\ref{hpinfformula}, this finishes the computation of the
local height $h_p(\infty_- - \infty_+,R-S)$.}
Note that for the computation of the integrals above we needed to assume
that $R$ and $S$ belong to the same residue disc or to \sg{non-Weierstrass} affine discs.
Also note that we do not need the cup product matrix to compute
$h_p(\infty_--\infty_+, R-S)$.

\sm{We emphasize that the only step in computing $h_p(\infty_--\infty_+, R-S)$ that depends on the points $R,S\in\Q_p$ is the computation of Coleman integrals with endpoints $R$ and $S$. Hence, if we want to compute several different local heights of the type $h_p(\infty_--\infty_+, R-S)$ (for example, for linear quadratic Chabauty in \cite{LinQC}), it is more efficient to complete all previous steps once and use them as input.}

\sg{\section{The affine case}\label{sec:affine}}

Our goal is to compute $h_p(P-Q,R-S)$ for four distinct affine points $P,Q,R,S\in
C(\Q_p)$. 
\sm{To this end, we describe Steps~\eqref{om'} to 
\eqref{ints} of Algorithm~\ref{alg:hp} in this case.} 
We assume that $P,Q,R,S$ satisfy \sg{Assumption~\ref{assumptions}}. 
As explained in~\S\ref{subsec:implementation}, our implementation in
\texttt{SageMath} assumes only Condition 1 for $P,Q,R,S$. 

We first use the same trick as in \cite{BBHeights}; we write a degree zero divisor as a sum of a symmetric and an antisymmetric one. 

\begin{lemma}\label{L:symm-antisymm}
\begin{align*} 
  h_p(P-Q,R-S) = &\dfrac{1}{2}\log_p\left(\dfrac{x(R)-x(P)}{x(R)-x(Q)}\cdot \dfrac{x(S)-x(Q)}{x(S)-x(R)}\right)\\ 
  &+ \dfrac{1}{2}h_p(P-\iota(P),R-S)-\dfrac{1}{2}h_p(Q-\iota(Q),R-S).
\end{align*}
\end{lemma}
\begin{proof}
From 
\[
\dv\left(\dfrac{x-x(P)}{x-x(Q)}\right)=P+\iota(P)-Q-\iota(Q)\,,
\]
we find
\[
P-Q=\dfrac{1}{2}\dv\left(\dfrac{x-x(P)}{x-x(Q)}\right)+\dfrac{1}{2}(P-\iota(P))-\dfrac{1}{2}(Q-\iota(Q))\,.
\]
Recall from Proposition \ref{prop:height-properties} that we have 
\[
h_p\left(\dv\left(\dfrac{x-x(P)}{x-x(Q)}\right),
  R-S\right)=\log_p\left(\dfrac{x(R)-x(P)}{x(R)-x(Q)}\cdot \dfrac{x(S)-x(Q)}{x(S)-x(P)}\right)\,.
\]
The claim follows from additivity. 
\end{proof}
Condition~1 for  $P,Q,R,S$ implies that the right hand side in
Lemma \ref{L:symm-antisymm} is defined. 
 However, the much weaker condition $\{x(R),x(S)\}\cap \{x(P), x(Q)\} =
\emptyset$ actually suffices. This is precisely Condition~1'. We will explain in~\S\ref{subsec:weakening} how to
weaken Condition 1 to Condition~1' in such a way that
Lemma \ref{L:symm-antisymm} continues to hold (though we did not implement
this). 
\sm{In fact, Lemma \ref{L:symm-antisymm} is the reason why we cannot weaken Condition
1', even in theory. }

From now on, we will restrict attention to $h_p(P-\iota(P), R-S)$,
where $P$ is a non-Weierstrass point.
In this case, {\it Step \eqref{om'}}  follows from the following result,
which can be viewed as an explicit version of \cite[Proposition
5.13]{BBHeights}.
\begin{proposition}\label{prop:antisymmetric-residue-differential}
The differential form 
$$\omega'=\dfrac{y(P)}{x-x(P)}\dfrac{dx}{y}$$ 
satisfies $\Res(\omega')=P-\iota(P)$.
\end{proposition}

\begin{proof}
To prove this, we use that the differential $\dfrac{dx}{y}$ is holomorphic on $C$. The only possible poles of $\omega'$ are at $P$ and $\iota(P)$. 
At these points, $t=x-x(P)$ is a uniformizer, so
$$\omega'=\dfrac{y(P)}{y}\dfrac{dt}{t},$$
which immediately gives $\Res_P(\omega')=1$ and $\Res_{\iota(P)}(\omega')=-1$.
\end{proof}

For {\it Step \eqref{psiom'}}  we follow and improve \cite[\S5.2]{BBHeights}.
Let $\omega'$ be a form as in Proposition \ref{prop:antisymmetric-residue-differential} and write
$\psi(\omega')=\sum_{i=0}^{2g-1}u_i\eta_i$. Then Proposition~\ref{prop:Cup-of-psi} implies
$$\langle\omega',\eta_j \rangle=\psi(\omega')\cup
[\eta_j]=\sum_{i=0}^{2g-1}u_i([\eta_i]\cup[\eta_j])\,.$$
Recalling that $M$ is the cup product matrix, we have
  \begin{equation}\label{cupsymbol}
\begin{pmatrix}
u_0 & u_1 & \cdots & u_{2g-1}
\end{pmatrix}^t =-M^{-1}\begin{pmatrix}
    \langle \omega',\eta_0\rangle  &
   \langle \omega',\eta_1\rangle  &
   \cdots &
    \langle \omega',\eta_{2g-1}\rangle 
\end{pmatrix}^t.
  \end{equation}
By the same argument as in \cite[Proposition 5.12]{BBHeights}, we have
\[
\langle\omega',\eta_j\rangle=-\int_{\iota(P)}^P\eta_j-\Res_{\infty_+}\left(\omega'\int\eta_j\right)-\Res_{\infty_-}\left(\omega'\int\eta_j\right)\,,
\]
since $\infty_{\pm}\notin\supp(\Res(\omega'))$.
Because $P$ and $\iota(P)$ (as well as $Q$ and $\iota(Q)$) are endpoints of
Coleman integrals, we use our assumption that
Condition~2 is satisfied for $P$ and, by symmetry, $Q$.
We can simplify this
further as follows:

\begin{lemma}\label{L:zero-residue-even}
\sg{Let $\omega'$ be a form as in Proposition \ref{prop:antisymmetric-residue-differential}.} We have $\Res_{\infty_+}\left(\omega'\int\eta_j\right)=\Res_{\infty_-}\left(\omega'\int\eta_j\right)=0$ for $0\leq j\leq 2g-1$.
\end{lemma}
\begin{proof}
It suffices to prove the statement for $\infty_+$ and for $j\geq g$.
  Consider the uniformizer $t=\frac{1}{x}$ at $\infty_+$. Then we have $\ord_t(y)=-g-1$. Using $dx=-t^{-2}dt$ we compute
\[
\ord_t\left(\dfrac{\omega'}{dt}\right)=\ord_t\left(-\dfrac{y(P)}{t^{-1}-x(P)}\dfrac{t^{-2}}{y}\right)=
  -2-(-1+(-g-1))=g,\quad \text{and}\] 

\[\ord_t\left(\int\eta_{j}\right)
  =\ord_t\left(\int \dfrac{t^{-j-1}t^{-2}dt}{t^{-g-1}}\right)= \ord_t\left(\int  t^{g-j-2}dt\right)=g-j-1.\]

Thus 
\[
\ord_t\left(\dfrac{\omega'\int\eta_{j}}{dt}\right)=2g-1-j\,,
\]
so $\omega'\int\eta_j$ is holomorphic at $\infty_+$ because $j\leq 2g-1$.
\end{proof}

\begin{remark}\label{R:zero-residue-odd}
We can prove a similar statement in the odd degree case. Namely, in the
  setting of \cite[Proposition 5.12]{BBHeights}, we also have
  $\Res_{\infty}\left(\omega'\int\eta_j\right)=0$ for $0\leq j\leq 2g-1$ \sg{and for the same definition of $\omega'$}.
  This can be used to simplify the computation of the global symbols
  in~\cite[\S5.2]{BBHeights}.
\end{remark}

\sg{We compute $u_0,\ldots,u_{g-1}\in \Q_p$ in Step \eqref{coeffs} as in
Lemma~\ref{lm:represent-psi-in-mixed-basis}.}

The main task of {\it Step \eqref{ints}} is the computation of the integral 
$$\int_S^R
\dfrac{y(P)}{x-x(P)}\dfrac{dx}{y}\,.$$ 
If $R$ and $S$ are in the same residue disc, then the integral
is tiny and can be computed directly. In the remaining cases, it can be computed using a change of variables that translates the desired
integral to the integral $\int \omega_g$ 
considered in Section~\ref{sec:Even-two-infinities}.
Namely, consider
\begin{align}\label{eq:change-of-variables}
  \tau\colon C&\to C'\colon
  y'^2=\dfrac{1}{y(P)^2}x'^{2g+2}f\left(x(P)+\frac{1}{x'}\right)\\
  (x,y)&\mapsto(x',y')\colonequals\left(\dfrac{1}{x-x(P)},\dfrac{-y}{y(P)(x-x(P))^{g+1}}\right)\,.\nonumber
\end{align}

\sg{Since $C'$ is defined by an even degree model, and since
$\tau(P)=\infty_-\in C'$ and $\tau(\iota(P))=\infty_+\in C'$, the
differential $\frac{y(P)}{x-x(P)}\frac{dx}{y}$ is mapped to
$\frac{x'^gdx'}{y'}$. This is also consistent with the fact that $\Res(\frac{y(P)}{x-x(P)}\frac{dx}{y})=P-\iota(P)$ and $\Res(\frac{x'^gdx'}{y'})=\infty_- - \infty_+$.}

\begin{lemma}\label{L:change}
Denote $R'=\tau(R)$, $S'=\tau(S)$. Then we have 
\[
  h_p(P-\iota(P),R-S) = \int_{S'}^{R'} \omega_g - u_0\int_S^R \eta_0 -
  \cdots - u_{g-1}\int_S^R \eta_{g-1}.
\]
\end{lemma}
\begin{proof}
  This follows from 
\begin{equation}\label{Integral:ChangeOfVariables}
\int_S^R \dfrac{y(P)}{x-x(P)}\dfrac{dx}{y} = 
\int_{S'}^{R'} \dfrac{x'^gdx'}{y'}\,, 
\end{equation}
for which we use the change of variables formula for
  Coleman integrals  (see \cite[Section~2]{Coleman-deShalit}.
\end{proof}
Lemma~\ref{L:change} reduces Step \eqref{ints} to Step \eqref{ints}
from \sg{Section~\ref{sec:Even-two-infinities}}. 

\begin{remark}\label{stepivcond2}
\sm{
In Lemma \ref{L:change}, the points $R, S\in
C(\Q_p)$, and $R',S'\in C'(\Q_p)$ are endpoints of Coleman integrals. 
Hence we
require Condition 2 \sg{for $R,S$}, and we require 
Condition
\sg{1} for $P,Q,R,S$. The division by $y(P)$ in the
definition of the map $\tau$ makes the assumptions $p\nmid y(P)$ and $p\nmid y(Q)$
necessary, since otherwise $C'$ could have non-integral coefficients, so we keep Condition 2 also for $P,Q$. }
\end{remark}

\sm{For convenience, we summarize our algorithm to compute
$h_p(P-Q, R-S)$, where all points are affine. This makes
Algorithm~\ref{alg:hp} more precise in this case.}
\sg{
\begin{alg} \label{alg:hp-affine}  {\bf (Computation of $h_p(P-Q, R-S)$ for
  affine $P,Q,R,S$)}
\hfill

 {\bf Input:} The same input as for Algorithm~\ref{alg:hp}, with all points
      $P,Q,R,S$ affine and satisfying Assumption~\ref{assumptions}.

{\bf Output:} The local height $h_p(P-Q, R-S)$.
\begin{enumerate}
\item Find a differential $\omega'$ such that $\Res(\omega')=P-\iota(P)$. Define  $\omega'$ as  in Proposition~\ref{prop:antisymmetric-residue-differential}.
  \item\label{affine:psi} Compute $\psi(\omega')$.
\item  Compute the unique coefficients $u_0,\ldots,u_{g-1}\in
  \Q_p$ 
    such that
   \sg{$\omega\colonequals \omega' -\sum^{g-1}_{i=0}u_i\eta_i$} satisfies $\Res(\omega)=P-Q$ and $\psi(\omega)\in W$.
\item  Compute the Coleman integrals $\int_S ^R \omega'$ and \sg{$\int_S ^R \eta_i$} for
  $i=0,\ldots,g-1$ and compute $h_p(P-\iota(P),R-S) =\int^R_S\omega'-\sum^{g-1}_{i=0}u_i\int^R_S\eta_i$.
\item Similarly, compute $h_p(Q-\iota(Q),R-S)$.
\item\label{affine:final-formula} Return
$$\dfrac{1}{2}\left(h_p(P-\iota(P),R-S)-h_p(Q-\iota(Q),R-S) +
    \log_p\left(\dfrac{x(R)-x(P)}{x(R)-x(Q)}\cdot \dfrac{x(S)-x(Q)}{x(S)-x(P)}\right)\right).$$
\end{enumerate}
\end{alg}}

\begin{remark}\label{R:BB-nonholo-integral}
Alternatively, one could potentially compute $\int^R_S\omega'$ using 
\cite[Algorithm 4.8]{BBHeights}. Define $\alpha\colonequals
  \phi^*\omega'-p\omega'$ and $\beta$ to be a differential whose residue
  divisor is $R-S$. Then the desired integral can be expressed as 
  \begin{equation}\label{E:4.8}
\int_S^R\omega'=\dfrac{1}{1-p}\cdot \left(\psi(\alpha)\cup\psi(\beta)+\sum_{P\in\mathcal{P}}\Res_P\left(\alpha\int\beta\right)-\int_{\phi(S)}^S\omega-\int_{R}^{\phi(R)}\omega\right),
  \end{equation}
where $\mathcal{P}$ is the subset of $C(\overline{\Q}_p)$ consisting of the
Weierstrass points and the
poles of $\alpha$. See Equation~(14) in~\cite{BBHeights}.
While this formula is only proved there for odd degree hyperelliptic curves, the extension
to even degree is immediate. This leads to a more general, but more complicated algorithm.
\end{remark}

\begin{remark}\label{R:aff_inf}
We can also completely reduce the affine case
  to the infinite case discussed in Section~\ref{sec:Even-two-infinities}. 
  Recall that it suffices to compute heights of the form $h_p(P-\iota(P), R-S)$.
  \sg{Using the same notation and the map $\tau\colon C\rightarrow C'$ as in Lemma
  \ref{L:change}, and the model-independence of local heights (see Corollary
  \ref{cor:height-independence}), this amounts to computing $h_p(\infty_- - \infty_+, R'-S')$
  on $C'$.}
\end{remark}

\section{Precision analysis}\label{sec:prec-p-adic-heights}
\sm{
  In this section we analyze the precision of our algorithms.}
  The algorithms in \S\ref{basis} and \S\ref{cup}  depend on the precision of local
coordinates. In both cases, we need to compute residues of differentials at
points at infinity. The computations in~\S\ref{cup} clearly require more precision than
\S\ref{basis}. We want to compute
$\Res_{\infty_-}\left(\eta_j\int\eta_i\right)$ for $0\leq i,j\leq 2g-1$.
Recall that $t=\frac{1}{x}$ is a uniformizer at $\infty_-$ and that for $g\leq i\leq 2g-1$ we have
$\eta_i=\left( \frac{1}{t^{i-g+2}}+\cdots\right)dt$. The term that requires the largest precision is 
\[
\Res_{\infty_-}\left(\eta_{2g-1}\int\eta_{2g-1}\right)=\left( \frac{1}{t^{g+1}}+\cdots\right)dt\cdot \left( \frac{1}{gt^{g}}+\cdots\right).
\]
Hence we need sufficient $t$-adic precision to compute the coefficient of
$t^{g-1}$ of $\eta_{2g}$, which amounts to computing the first $2g+1$ coefficients of $y=y(t)$. 

Using Harrison's extension of Kedlaya's algorithm~\cite{Harrison-Even-Deg-MW}, we can
compute the matrix of Frobenius acting on $[\omega_0],\ldots,[\omega_{2g}]\subset \hmw(U)^-$ to any desired precision $p^n$. We assume
that we have already computed the basis differentials
$\eta_0,\ldots,\eta_{2g-1}$ in~\S\ref{basis} to the same
precision $p^n$. Then it is clear that the results of~\S\ref{frob}
and~\S\ref{subsec:subspace} are correct to
precision $p^n$ as well. 

\sm{We now analyze precision for Algorithm}~\ref{alg:hp}.
Step \eqref{om'} is exact, so there is no loss of precision. 
\sg{
Step \eqref{coeffs} only involves a linear change of variables by Lemma~\ref{lm:represent-psi-in-mixed-basis}, and the loss of precision in this step is bounded by the valuation of the determinant of the corresponding matrix, which can be precomputed.


The precision in Step \eqref{ints} depends on the precision of the Coleman integrals and the coefficients computed in \eqref{coeffs}. We compute the integrals using Algorithm 4 from \cite{Jen-Even-Degree-CI}. As stated in loc. cit., the same precision estimates as in~\cite{BBK} apply. Briefly:
\begin{itemize}
\item[(a)] By~\cite[Proposition~18]{BBK}, the tiny integrals in~\cite[Equation~(3.2)]{Jen-Even-Degree-CI} are computed correctly to $\min\{n,k+1-\lfloor\log_p(k+1)\rfloor$ digits of precision, where the points are accurate to $n$ digits and we truncate the local expansions of the differentials at precision $t^k$.

\item[(b)] As in~\cite{BBK}, the evaluation of the functions $h_i$ at the
  endpoints does not lead to any loss in precision. Indeed, according to
    Harrison's algorithm, the functions $h_i$ are finite sums of terms
    $c_{i,j}x^iy^j$, where $i\in \Z, i\geq 0$, $j\in \Z$. If the $c_{i,j}$
    are computed modulo $p^n$, then $h_i(P)$ is correct to at least
    the same precision.
\item[(c)] By~\cite[Proposition~19]{BBK}, if we denote $m'\colonequals \ord_p\det(\Phi-I)$, and the previous computations are correct modulo $p^n$, then the Coleman integrals are computed correctly modulo $p^{n-m'}$.
\end{itemize}
\begin{remark}
    By~\cite[Section~3]{Harrison-Even-Deg-MW}, a working precision of $n+\lfloor\log_p(2n)\rfloor+1$ digits suffices to compute both the matrix of Frobenius and the functions $h_i$ to precision $p^n$.
\end{remark}

Step \eqref{psiom'} is the only step where we need to distinguish cases. We first deal with the approach from \sg{Section~\ref{sec:Even-two-infinities}}. Suppose that the matrix of Frobenius and the differentials $\eta_0,\ldots,\eta_{2g-1}$ of
$\hdr$ are correct modulo $p^n$. Denoting  $m\colonequals \ord_p\det(\Frob-pI)$, we compute $\psi(\omega')$ in terms of $\eta_0,\ldots,\eta_{2g-1}$ correctly modulo $p^{n-m}$.

It remains to analyze the precision for Step \eqref{psiom'} in \sg{Section~\ref{sec:affine}}; more precisely, this is Step~\eqref{affine:psi} in Algorithm~\ref{alg:hp-affine}). We first compute the global symbols $\langle
\omega',  \eta_j\rangle = -\int_{\iota(P)}^P\eta_j$.
and then apply~\eqref{cupsymbol}.
Hence, if all computations so far are correct modulo $p^n$, and we denote $m''=\ord_p\det(M)$, where $M$ is the cup product matrix, the coordinates of $\psi(\omega')$ in Step \eqref{psiom'} are computed accurately modulo $p^{n-m''}$.

All three matrices $\Frob-pI$, $\Phi-I$, and $M$ depend only on the curve. Thus, we may precompute their determinants and predict which starting
precision is necessary to reach a specified target precision.
For the approach discussed in Remark~\ref{R:aff_inf}, we also need to
compute $\det(\Frob-pI)$ and $\det(\Phi-I)$ for the curve $C'$, which depends on the point $P$.

In addition, Step~\eqref{affine:final-formula} in Algorithm~\ref{alg:hp-affine} requires an evaluation of $\log_p$, but since $P,Q,R,S$ satisfy Assumption~\ref{assumptions}, this does not lead to any loss of precision. 

\begin{remark}\label{R:pg}
Our precision analysis assumes that the matrix of Frobenius
on $\hmw(U)$ is $p$-adically integral. 
As shown in~\cite[Lemma~3.4, Conclusion]{Harrison-Even-Deg-MW} this is the
case when $p>g$, so we assume this, as
in \cite{Jen-Even-Degree-CI}, whenever needed.
\end{remark}

\begin{remark}\label{R:precspecial}

  In the precision analysis above, we assume Assumption~\ref{assumptions}. As discussed
  in~\S\ref{subsec:weierstrass} and \S\ref{subsec:inf},
  Condition~2 is not necessary if we assume Condition~1, and in order to reduce 
  to the situation treated above, only symmetry and 
tiny integrals are required.
  Hence, we do not need an additional
  precision analysis in this more general situation.
\end{remark}
}

\section{Implementation}\label{subsec:implementation}
We have implemented Algorithm~\ref{alg:hp} in {\tt SageMath}. For simplicity,
we assume that $p>g$, whenever needed (see Remark~\ref{R:pg}),
and that the branch $\log_p$ is the
branch determined by $\log_p(p)=0$. This suffices for our applications
discussed in the present article and in~\cite{LinQC}.
{\tt SageMath} is particularly suitable for our approach because there are
existing {\tt SageMath}-implementations of 
\begin{itemize}
\item Harrison's extension of Kedlaya's algorithm to
    compute $\hmw(U)$ and the action of Frobenius on $\hmw(U)^-$
    (see~\cite{Harrison-Even-Deg-MW}) and
\item Balakrishnan's algorithm from~\cite{Jen-Even-Degree-CI} to compute
  Coleman integrals
\end{itemize}
when $C$ is a monic even degree hyperelliptic curve; 
  available from~\url{https://github.com/jbalakrishnan/AWS}.
 In particular, the command \verb+coleman_integrals_on_basis+ computes the integrals
$\int_S^R\eta_0,\ldots,\int_S^R\eta_{2g}$ efficiently if $R$ and $S$ are $\Q_p$-points in
affine discs. 
This made it easy to implement the algorithm outlined 
in Sections~\ref{sec:Even-two-infinities} and~\ref{sec:affine}
for points $P,Q,R,S\in C(\Q_p)$
that satisfy \sg{Assumption~\ref{assumptions}}. Recall that by~\S\ref{subsec:weierstrass} we can drop Condition~2 \sg{for both pairs $(P,Q)$ and $(R,S)$} provided
Condition~1 is satisfied.

We require Condition~1 because we use Lemma~\ref{L:change} for
Step \eqref{ints}
of Section~\ref{sec:affine}, since it is much simpler than the one based on
Equation~\eqref{E:4.8}. The latter would require
an implementation of a generalization of the rather complicated Algorithm~4.8 from~\cite{BBHeights} to even
degree, which we did not attempt. In particular, the computation of the residues in~\eqref{E:4.8} requires
computations to high precision in annuli with Laurent series that have an essential
singularity and computations over extensions of $\Q_p$ that can have degree as large as
$p$, or even larger.
While this would lead to a more general implementation than our present one, we expect it to be  
significantly slower.
    We have also implemented the approach discussed in
    Remark~\ref{R:aff_inf} in {\tt SageMath}. For a single  height computation,
    its performance is similar to that of the implementation discussed in
    the present subsection. However, when we need to compute several local
    heights on the same curve, then the approach from
    Remark~\ref{R:aff_inf} requires precomputations for several different
    curves that show up in this approach, making it less efficient.

In practical applications, such as quadratic Chabauty or the computation of the $p$-adic
regulator, we often have some freedom in choosing representatives, so that an algorithm
for points satisfying Condition~1 usually suffices.

\subsection{Timings}\label{subsec:timings}
To illustrate the performance of our algorithm we present timings in
Table~\ref{tab:timings} by computing $h_p(P-Q, R-S)$ for the following curves $C_g$ of respective genus
$g$ and various primes $p$:
\begin{align*}
  C_2&\colon y^2=x^5+(x-1)(x-4)(x-9)(x-16)\\
    &P=(1,1),\, Q=(4,2^5),\,R=(9,3^5),\,S=(16,4^5) \in C_2(\Q_p)\\
  C_3&\colon y^2=x^7+(x-1)^2(x-4)^2(x-9)(x-16)\\ 
    &P=(1,1),\, Q=(4,2^7),\,R=(9,3^7),\,S=(16,4^7) \in C_3(\Q_p)\\
  C_4&\colon y^2=x^9+(x-1)^2(x-4)^2(x-9)^2(x-16)^2\\ 
    &P=(1,1),\, Q=(4,2^9),\,R=(9,3^9),\,S=(16,4^9) \in C_4(\Q_p)\\
  C_{17}&\colon y^2=x^{35}+(x-1)^2(x-4)^2(x-9)^2(x-16)^2\\
    &P=(1,1),\, Q=(4,2^{35}),\,R=(9,3^{35}),\,S=(16,4^{35}) \in C_{17}(\Q_p)
\end{align*}
For example, we find that on $C_2$ we have
\[h_7(P-Q,R-S)= 5\cdot7 + 4\cdot7^3 + 6\cdot7^4 + 3\cdot7^5 + 3\cdot7^6 +
4\cdot7^7 + 4\cdot7^8 + 6\cdot7^9 + O(7^{10})\,.
\]

\begin{table}[t]
\begin{tabular}{| l | c | c | c | c | }
\hline
  Curve & $p$ & Precision & Our time & \cite{BBHeights} \\
\hline 
$C_2$ & 7 & 10 & 2s & 7s  \\ 
$C_2$ & 7 & 300 & 11m &  \\ 
$C_2$ & 503 & 10 & 4m & 19h \\ 
$C_3$ & 11 & 10 & 6s & 28s \\
$C_4$ & 23 & 20 & 2m & 46m \\
$C_{17}$ & 11 & 7 & 14m & \\ \hline
\end{tabular}
  \caption{Timings}\label{tab:timings}
\end{table}
The examples were computed on a single core of a 4-core 2.6 GHz Intel
i7-6600U CPU with 8GB RAM, running {\tt SageMath~9.8} on {\tt  Xubuntu 20.04}.
Whenever possible, we compare our
implementation against the implementation of 
the algorithm 
from~\cite{BBHeights} in \texttt{SageMath} (see
  \url{https://github.com/jbalakrishnan/AWS}).
We found that our implementation performed better in all examples. 
It can handle fairly large genera, primes and
precisions, whereas the \texttt{SageMath}-implementation  
of the algorithm from~\cite{BBHeights} is limited to genus at most~4 and
did not finish the computation of $h_{7}(P-Q,R-S)$ on $C_2$ to $300$ digits
of precision in one week.
Code for these timings can be found at~{\url{https://github.com/StevanGajovic/heights_above_p}}.

\begin{remark}\label{R:extensions}
  The main reason for assuming that our points are defined over $\Q_p$ is
  that for even degree hyperelliptic curves, no implementation of Coleman
  integration over proper extensions of $\Q_p$ is currently available. Once such
  an implementation becomes available, it should be possible to extend 
  our algorithm without major problems.
    For instance, Best~\cite{Alex-ColemanIntegration-Unramified-Superelliptic} has implemented Coleman integration for superelliptic
    curves $y^n=f(x)$ with $\gcd(n,\deg(f))=1$ in {\tt Julia}, extending the approach of
    Balakrishnan--Bradshaw--Kedlaya~\cite{BBK} for odd degree hyperelliptic curves. Notably, his implementation
    works for unramified extensions of $\Q_p$. The special case of an extension to
    $\gcd(n,\deg(f))\ne 1$ would allow us to extend our algorithm in practice to points
    $P,Q,R,S$ defined over unramified extensions of $\Q_p$.
    Moreover, in forthcoming work, Best, Kaya and Keller will extend the general
    Coleman integration algorithm from~\cite{BalakrishnanTuitman} and give
    a {\tt Magma}-implementation, extending the one
for Coleman integrals over $\Q_p$ available 
    from~{\url{https://github.com/jtuitman/Coleman}}.
\end{remark}

\subsection{Implementation - special cases}\label{subsec:implementation_special_cases}

\sg{We now explain how to compute $h_p(P-Q, R-S)$ when Condition 2 is not
satisfied for either $(P,Q)$ or $(R,S)$}. Recall that we still assume that $P,Q,R,S\in C(\Q_p)$ satisfy Condition 1.

\subsubsection{Weierstrass and infinite discs}\label{subsec:weierstrass}
We can deal with the case that $R$ or $S$ belongs to a Weierstrass disc as
follows. Suppose that $D(R)$ contains a Weierstrass point $T$ and that
$\omega$ has no singularity at $T$. This is satisfied in our applications of Algorithm~\ref{alg:hp}.
Then
$\int_S^R \omega= \int_S^T \omega + \int_T^R \omega$. The second integral is tiny,
   so we can compute it. Since we integrate odd forms, we can also compute the first integral, using $\int_S^T \omega=\frac{1}{2}\int_S^{\iota(S)} \omega$ (see \cite[Lemma 16]{BBK}). 
 We could therefore replace 
  Condition~2 by the weaker condition that
   $\ord_p(y(R))\geq 0$ and $\ord_p(y(S))\geq 0$.

By symmetry, it follows that we can compute $h_p(P-Q,R-S)$ if
$\ord_p(y(P)),\ord_p(y(Q))\geq 0$ or if  $\ord_p(y(R)),\ord_p(y(S))\geq 0$.
Lemma \ref{L:symm-antisymm} implies that it is enough to compute heights of
the  type $h_p(P-\iota(P),R-\iota(R))$. Condition 1 implies that
it is not possible to have both $\ord_p(y(P))<0$ and $ \ord_p(y(R)) <0$ at
the same time. Hence, without loss of generality, the only remaining case
is when we have $\ord_p(y(P))\ge 0$, $\ord_p(y(R)) <0$, and $R\in D(\infty_-)$. Then
\begin{multline*}
h_p(P-\iota(P),R-\iota(R))=h_p(P-\iota(P),\infty_- - \infty_+)\\ +
  h_p(P-\iota(P), R - \infty_-) -  h_p(P-\iota(P), \iota(R) - \infty_+)\,.
\end{multline*}
We find $h_p(P-\iota(P),\infty_- - \infty_+)$ using the method from
Section~\ref{sec:Even-two-infinities}. The integrals that show up in Step
\eqref{psiom'} and Step \eqref{ints} when computing $h_p(P-\iota(P), R -  \infty_-)$ and $h_p(P-\iota(P), \iota(R) - \infty_+)$ are tiny, so we can  compute all heights on the right hand side.

 Hence we only have to
assume Condition 1 when the points $P,Q,R,S$ are affine. 
We have implemented this approach in \texttt{SageMath}. 
\subsubsection{Points at infinity}\label{subsec:inf}
When exactly one of the points $P,Q,R,S$ (satisfying Condition 1) is a point at infinity, we compute
  the local height using a similar strategy. Namely, we have
  \begin{align*}
    h_p(\infty_- - Q,R-S)
     =&\dfrac{1}{2}\left(\log_p\left(\dfrac{x(S)-x(Q)}{x(R)-x(Q)}\right)\right.
     \\&+ \left.h_p(\infty_- - \infty_+,R-S)-  h_p(Q-\iota(Q),R-S) \right).
  \end{align*}
The case $P=\infty_+$ is analogous, since $h_p(\infty_+ - Q,R-S)=h_p(\infty_- -\iota(Q),\iota(R)-\iota(S))$. 
We have also implemented this case in \texttt{SageMath}.

\subsubsection{Weakening Condition~1 to Condition~1'}\label{subsec:weakening}
We now explain how Condition~1 may be weakened. However, we have
not implemented this, so in our implementation 
we assume Condition~1.

Assume that the points $P,Q,R,S$ do not satisfy Condition~1, but do
satisfy Condition~1'.
By Remark \ref{R:aff_inf}, it is enough to consider the approach in
Section~\ref{sec:Even-two-infinities}, where Condition~1
is required to perform Step~\eqref{ints}. 
One possible approach is to use the existing {\tt Magma} implementation of the algorithm
for Coleman integrals on general 
curves due to Balakrishnan and
    Tuitman~\cite{BalakrishnanTuitman}. This algorithm directly computes Coleman integrals with
    respect to a basis of $\hdr(C/\Q_p)$. However, it currently cannot be used
    to compute integrals of differentials that are not essentially of the second kind; in particular,
    we cannot compute integrals of the form $\int^R_S\omega_g$ directly.  Instead, we could
    follow the strategy from~\cite{BBHeights} to carry out Step~\eqref{ints}, by using \sg{the differential $\alpha=\phi^{*}(\omega_g)-p\omega_g$. which is essentially of the second kind.}  Hence
    $\int_R^S \alpha$ can be computed using the existing {\tt Magma} implementation even when $R$ or $S$ belong to $D(\infty_-)$
    or $D(\infty_+)$.
    We obtain the required integral from
    \[
    \int_S^R\omega_g=\dfrac{1}{1-p} \left(\int_S^R\alpha - \int_{\phi(S)}^S\omega_g - \int_R^{\phi(R)}\omega_g\right).
    \]
The integrals  \sg{$ \int_{\phi(S)}^S\omega_g$} and \sg{$\int_R^{\phi(R)}\omega_g$} are tiny integrals,
  hence they can be computed easily. 

  \begin{remark}\label{R:} 
As remarked in Step \eqref{ints} of Section~\ref{sec:affine}, if the
endpoints of a Coleman integral are in the same residue disc, we can
compute it as a tiny integral. Hence it would be also possible to weaken
Condition~1 by allowing that $R$ and $S$ (or by symmetry, $P$ and
$Q$) are in the same residue disc. But as this is a minor improvement
comparing to weakening Condition~1 to Condition~1' and
was not necessary for our applications, we have not implemented this.
  \end{remark}

\section{Numerical evidence for $p$-adic BSD}\label{S:pbsd}

For elliptic curves $E/\Q$ of good ordinary reduction at a prime $p$,
Mazur--Tate--Teitelbaum propose 
\footnote{In fact, Mazur--Tate--Teitelbaum also give conjectures for bad
semistable reduction; a supersingular conjecture is due to Bernardi and
Perrin--Riou~\cite{BPR93}.} 
a $p$-adic conjecture of Birch and
Swinnerton--Dyer type in~\cite{MTT86}. It relates the special value of the 
analytic $p$-adic $L$-function $L_p(E,s)$ of
Mazur and Swinnerton--Dyer~\cite{MSD74} to the $p$-adic regulator
defined via the canonical $p$-adic height of Mazur--Tate from \cite{Mazur-Tate-p-adic-heights}. They also give numerical
evidence for their conjecture; for further evidence see~\cite{SW13}.
Together with Balakrishnan and Stein, the second author has
extended this conjecture
in~\cite{BMS16} to the case of an abelian variety $A/\Q$ with good ordinary
reduction at a prime $p$ such that $A$ is modular
of $\mathrm{GL}_2$-type, associated to $d=\dim(A)$ conjugate newforms
$f_1,\ldots, f_d \in S_2(\Gamma_0(N))$. 

To formulate the conjecture, we recall some notation
from~\cite{BMS16}.
Let $\Reg_p(A/\Q)$ denote the $p$-adic regulator of $A(\Q)$, defined via
the canonical Mazur--Tate height and the cyclotomic idèle class character,
normalized as in~\cite[Definition~3.3]{BMS16}.
Let $L_p(A,s)\colonequals L_p(f_1,s)\cdots L_p(f_d,s)$, where the
Mazur--Swinnerton--Dyer $p$-adic $L$-functions $L_p(f_i,
s)$ are normalized with respect to Shimura periods that
satisfy~\cite[Theorem~2.4]{BMS16}. 
Let $K$ denote the (totally real)
field generated by the Hecke eigenvalues $a_{p,i}$ of the $f_i$ and fix an embedding
$\sigma\colon K\hookrightarrow\C$ and a prime $\p\mid p$ of $K$. Let $\alpha_i$ be the unit root of 
$x^2-\sigma(a_{p,i})x+p \in K_\p[x]$ and define 
the $p$-adic multiplier $\epsilon_p(A)\colonequals \prod_{i=1}^d(1-\alpha^{-1}_i)^2$.  
Finally, let $c_v$ denote the
    Tamagawa number at a finite place $v$ of $\Q$, $\Sha(A/\Q)$ the
    Shafarevich--Tate
group of $A/\Q$ and $A(\Q)_{\tors}$ the torsion subgroup of $A(\Q)$.
Then a slightly rewritten and specialized version
of~\cite[Conjecture~1.4]{BMS16} reads as follows:

\begin{conj}\label{pbsd}
Let $A/\Q$ be a modular abelian variety of  $\mathrm{GL}_2$-type, with
good ordinary reduction at a prime $p$.
Then the Mordell--Weil rank $r$ of $A/\Q$ equals
      $\ord_{s=1}L_p(A,s)$ and we have
      \begin{equation*}\label{eq-pbsd}{L}_p^*(A,1)
        =\frac{\epsilon_p(A)}{\log_p(1+p)^r}
\cdot\frac{ |\Sha(A/\Q)|
      \cdot \Reg_{p}(A/\Q) \cdot \prod_v  c_v
}{|A(\Q)_{\tors}|\cdot|A^{\vee}(\Q)_{\tors}|},\end{equation*} where
  $L_p^*(A,1)$
is the leading coefficient of $L_p(A,s)$ at $s=1$. \end{conj}

We now discuss how to gather empirical evidence for Conjecture~\ref{pbsd}. The $p$-adic
$L$-functions $L_p(f_i,s)$ can be approximated using an algorithm due to
Pollack--Stevens (see~\cite{PS11}) based on overconvergent modular
symbols. 
{\tt SageMath} contains an implementation of this algorithm;
currently, it requires that the prime $p$ splits in $K$. 
A forthcoming {\tt Magma}
implementation of this algorithm that does not have this restriction is due
to Keller.
The Tamagawa numbers can be computed using work of van Bommel~\cite{vB22}.
 In~\cite{KS22, KS} Keller and Stoll announce algorithms to compute
 $|\Sha(A/\Q)|$ and the verification that $|\Sha(A/\Q)|=1$ for a
 number of geometrically simple abelian surfaces.
However, if we cannot compute $|\Sha(A/\Q)|$, we can still
try to compute all other quantities in Conjecture~\ref{pbsd}
and test whether the quotient
\begin{equation*}\label{}
  \frac{{L}_p^*(A,1)\cdot \log_p(1+p)^r\cdot |A(\Q)_{\tors}|\cdot|A^{\vee}(\Q)_{\tors}|}
{\epsilon_p(A)
      \cdot \Reg_{p}(A/\Q) \cdot \prod_v  c_v
}
\end{equation*}
is an integer and whether it agrees with what we expect $|\Sha(A/\Q)|$ to
be.

Suppose that $A$ is the Jacobian of a smooth projective geometrically
integral curve $X/\Q$, satisfying the
conditions of Conjecture~\ref{pbsd}.
By \cite[Theorem 3.3.1]{Coleman-bi-extension-p-adic-heights}, the
canonical Mazur--Tate height equals the Coleman--Gross height with respect
to the unit root subspace (both taken with respect to the same idèle
class character).
If $X$ is a hyperelliptic curve given by an odd degree model over $\Q_p$,
then one can compute $\Reg_p(A/\Q)$ using the algorithm of
Balakrishnan--Besser from \cite{BBHeights} (and {\tt Magma}'s {\tt
LocalIntersectionData})
Using this approach, evidence for Conjecture~\ref{pbsd} is discussed
in~\cite[Section~4]{BMS16} for the Jacobians of~16 genus-2 curves having
rank~2 and for all good ordinary primes $p<100$ such that the curve has an odd degree model
over $\Q_p$.

Using our Algorithm~\ref{alg:hp}, we can gather empirical evidence for 
Conjecture~\ref{pbsd} for hyperelliptic curves $X/\Q$ with modular
Jacobian of GL$_2$-type, and good ordinary primes $p$ such that $X$ has no odd degree model
over $\Q_p$. 

\begin{example}\label{E:67bsd}
 Consider the Atkin--Lehner quotient $X\colonequals X_0^+(67)$,
given by the equation 
  \begin{equation*}\label{67eqn}
X\colon y^2 = x^6+4x^5+2x^4+2x^3+x^2-2x+1\,,
  \end{equation*}
and let $A$ be its Jacobian.
  It was shown in~\cite{FLSSSW01} that
  \begin{equation}\label{tametc}
    \mathrm{rk}(A/\Q)=2,\quad c_v=1\text{ for all }v,\quad
    |A(\Q)|_{\mathrm{tors}}=1
  \end{equation}
 and that the full (classical) Birch and Swinnerton--Dyer conjecture
holds numerically for $A/\Q$  if and only if $|\Sha(A/\Q)|=1$. The latter was recently 
verified by Keller and Stoll~\cite{KS22, KS}.
In~\cite{BMS16}, Conjecture~\ref{pbsd} was verified numerically for the
primes $p=7,13,17,19,37,41,43,47,59,61,73,79,83$. These are all good
ordinary primes $p<100$ such that $X$ has a quintic model over $\Q_p$.

  We now discuss the verification of Conjecture~\ref{pbsd} for $p=11$.
  The Fourier coefficients of the newforms corresponding to $A$ lie in
  $K=\Q(\sqrt{5})$. 
  Since $p=11$ splits over $K$, we can use the {\tt SageMath} implementation of
  the overconvergent algorithm of Pollack and Stevens to compute
  $L_{11}^*(A,1)$. According to~\cite[Table 4.4]{BMS16}, this involves a 
  normalization factor $\delta^+$, which is independent of $p$. In the present example, we have $\delta^+=1/4$ by~\cite[Table~4.4]{BMS16}. We find that
  \begin{equation}\label{l11}
    L_{11}^*(A,1) = 10 + 11 + 3\cdot 11^2 + 10\cdot 11^3 + O(11^4)\,.
  \end{equation}
  The $11$-adic multiplier is
  \begin{equation}\label{eps11}
  \epsilon_{11}(A)=9 + 8\cdot 11 + 11^2 + 5\cdot 11^3 + O(11^4)\,.
  \end{equation}
  It remains to compute $\Reg_{11}(A/\Q)$. The Mordell--Weil group $A(\Q)$
  is generated by $P_1 = [(0,1) - \infty_-]$ and $P_2 = [(0,1)-(0,-1)]$.
  Since $ x^6+4x^5+2x^4+2x^3+x^2-2x+1$ has no root in
  $\F_{11}$, we cannot use the algorithm from~\cite{BBHeights}.
  Instead, we apply Algorithm~\ref{alg:hp} to compute the $11$-adic heights
  $h(P_1,P_1), h(P_1,P_2)$ and $h(P_2,P_2)$; we have 
$$
  \Reg_{11}(A/\Q) = h(P_1,P_1)h(P_2,P_2)-h(P_1,P_2)^2\,.
$$
  To do so, we find suitable multiples of $P_1$ and $P_2$ 
that have representatives
  satisfying the
  conditions of Algorithm~\ref{alg:hp} and of the algorithm
  in~\cite{Mul14}.
  For instance, $3P_1$ and $10P_1$ have representatives $D_{1}$ and $D'_1$,
  respectively, such that 
  \begin{itemize}
    \item $D_1$ and $D'_1$ have disjoint support;
    \item $D_1\otimes \Q_{11}$ is of the form $Q_1+Q_2-R_1-R_2$, where
      $Q_1,Q_2\in X(\Q_{11})$ are in affine discs and $R_{1}, R_2$ are the
      points with $x$-coordinate~1; 
    \item $D'_1\otimes \Q_{11}$ is of the form $Q'_1+Q'_2-R'_1-R'_2$, where
      $Q'_1,Q'_2\in X(\Q_{11})$ are in affine discs and $R'_1, R'_2$ are the
      points with $x$-coordinate~2; 
    \item the tuples $(Q_i,R_i, Q'_j, R'_j)$ satisfy Condition~1
      for all $i,j\in \{1,2\}$.
  \end{itemize}
  Hence we may compute $h_{11}(D_1,D'_1)$ using Algorithm~\ref{alg:hp} and
  $h_{q}(D_1,D'_1)$ for $q\ne p$ using {\tt Magma}'s {\tt 
LocalIntersectionData}, and we find
$$
h(P_1, P_1) = \frac{1}{30}(h_{11}(D_1,D'_1) -\log_{11}(2)
-2\log_{11}(3)+2\log_{11}(5)+\log_{11}(17)-\log_{11}(317))\,.
$$
We compute $h(P_1,P_2)$ and $h(P_2,P_2)$ using the same strategy and obtain
$$
\Reg_{11}(A/\Q)=6\cdot 11^2 + 3\cdot 11^3 + 7\cdot 11^4 + O(11^5)\,.
$$
Together with~\eqref{tametc},~\eqref{l11} and~\eqref{eps11}, this suffices
to show that Conjecture~\ref{pbsd} holds to 4 digits of precision. 
In fact, we verified the conjecture to 10 digits of precision, and we
did the same for the primes $p=29,31,53,71,89$.
The code for these computations can be found
at~{\url{https://github.com/StevanGajovic/heights_above_p}}.

\end{example}

\bibliographystyle{alpha}
\bibliography{References}

\newcommand{\etalchar}[1]{$^{#1}$}
\begin{thebibliography}{BDM{\etalchar{+}}23}

\bibitem[AAB{\etalchar{+}}23]{AABCCKW}
Nikola Ad{\v{z}}aga, Vishal Arul, Lea Beneish, Mingjie Chen, Shiva Chidambaram, Timo Keller, and Boya Wen.
\newblock Quadratic {C}habauty for {A}tkin-{L}ehner quotients of modular curves of prime level and genus 4, 5, 6.
\newblock {\em Acta Arithmetica}, Published online, 2023.

\bibitem[Bal15]{Jen-Even-Degree-CI}
Jennifer~S. Balakrishnan.
\newblock Coleman integration for even-degree models of hyperelliptic curves.
\newblock {\em LMS J. Comput. Math.}, 18(1):258--265, 2015.

\bibitem[BB12]{BBHeights}
Jennifer~S. Balakrishnan and Amnon Besser.
\newblock Computing local {$p$}-adic height pairings on hyperelliptic curves.
\newblock {\em Int. Math. Res. Not.}, (11):2405--2444, 2012.

\bibitem[BBB{\etalchar{+}}21]{BBBLMTV19}
Jennifer\thinspace{}S. Balakrishnan, Alex~J. Best, Francesca Bianchi, Brian Lawrence, J.\thinspace{}Steffen M\"uller, Nicholas Triantafillou, and Jan Vonk.
\newblock Two recent {$p$}-adic approaches towards the (effective) {M}ordell conjecture.
\newblock In {\em Regulators {IV}: {A}n international conference on arithmetic {L}-functions and differential geometric methods}, volume 338 of {\em Progr. Math.}, pages 31--74. Birkh\"auser Boston, Boston, MA, 2021.

\bibitem[BBBM21]{BBBM-QCNF}
Jennifer~S. Balakrishnan, Amnon Besser, Francesca Bianchi, and J.~Steffen M\"{u}ller.
\newblock Explicit quadratic {C}habauty over number fields.
\newblock {\em Israel J. Math.}, 243(1):185--232, 2021.

\bibitem[BBK10]{BBK}
Jennifer~S. Balakrishnan, Robert~W. Bradshaw, and Kiran~S. Kedlaya.
\newblock Explicit {C}oleman integration for hyperelliptic curves.
\newblock In {\em Algorithmic number theory}, volume 6197 of {\em Lecture Notes in Comput. Sci.}, pages 16--31. Springer, Berlin, 2010.

\bibitem[BBM16]{BBM1}
Jennifer~S. Balakrishnan, Amnon Besser, and J.~Steffen Müller.
\newblock Quadratic {C}habauty: $p$-adic heights and integral points on hyperelliptic curves.
\newblock {\em Journal für die reine und angewandte Mathematik}, 2016(720):51 -- 79, 01 Nov. 2016.

\bibitem[BBM17]{BBM2}
Jennifer Balakrishnan, Amnon Besser, and J.~Steffen Müller.
\newblock Computing integral points on hyperelliptic curves using quadratic {C}habauty.
\newblock {\em Math. Comp.}, 86:1403--1434, 2017.

\bibitem[BC94]{Baldassarri-Chiarellotto}
Francesco Baldassarri and Bruno Chiarellotto.
\newblock Algebraic versus rigid cohomology with logarithmic coefficients.
\newblock In {\em Barsotti {S}ymposium in {A}lgebraic {G}eometry ({A}bano {T}erme, 1991)}, volume~15 of {\em Perspect. Math.}, pages 11--50. Academic Press, San Diego, CA, 1994.

\bibitem[BCP97]{Magma}
Wieb Bosma, John Cannon, and Catherine Playoust.
\newblock The {M}agma algebra system. {I}. {T}he user language.
\newblock {\em J. Symbolic Comput.}, 24(3--4):235--265, 1997.

\bibitem[BD18]{Jen-Netan-QCRP1}
Jennifer~S. Balakrishnan and Netan Dogra.
\newblock Quadratic {C}habauty and rational points, {I}: {$p$}-adic heights.
\newblock {\em Duke Math. J.}, 167(11):1981--2038, 2018.
\newblock With an appendix by J. Steffen M\"{u}ller.

\bibitem[BD20]{Jen-Netan-QCRP2}
Jennifer~S. Balakrishnan and Netan Dogra.
\newblock {Quadratic Chabauty and Rational Points II: Generalised Height Functions on Selmer Varieties}.
\newblock {\em International Mathematics Research Notices}, 02 2020.
\newblock rnz362.

\bibitem[BDM{\etalchar{+}}19]{QC13}
Jennifer~S. Balakrishnan, Netan Dogra, J.~Steffen M\"{u}ller, Jan Tuitman, and Jan Vonk.
\newblock Explicit {C}habauty-{K}im for the split {C}artan modular curve of level 13.
\newblock {\em Ann. of Math. (2)}, 189(3):885--944, 2019.

\bibitem[BDM{\etalchar{+}}23]{BDMTV2}
Jennifer~S. Balakrishnan, Netan Dogra, J.~Steffen M\"{u}ller, Jan Tuitman, and Jan Vonk.
\newblock Quadratic {C}habauty for modular curves: algorithms and examples.
\newblock {\em Compos. Math.}, 159(6):1111--1152, 2023.

\bibitem[Ber97]{Berthelot-isomorphism-rigid-MW}
Pierre Berthelot.
\newblock Finitude et puret\'{e} cohomologique en cohomologie rigide.
\newblock {\em Invent. Math.}, 128(2):329--377, 1997.
\newblock With an appendix in English by Aise Johan de Jong.

\bibitem[Bes00]{Besser-Syntomic-2}
Amnon Besser.
\newblock Syntomic regulators and {$p$}-adic integration. {II}. {$K_2$} of curves.
\newblock In {\em Proceedings of the {C}onference on {$p$}-adic {A}spects of the {T}heory of {A}utomorphic {R}epresentations ({J}erusalem, 1998)}, volume 120, pages 335--359, 2000.

\bibitem[Bes05]{Besser-padic-Arakelov}
Amnon Besser.
\newblock {$p$}-adic {A}rakelov theory.
\newblock {\em J. Number Theory}, 111(2):318--371, 2005.

\bibitem[Bes23]{Alex-ColemanIntegration-Unramified-Superelliptic}
Alex Best.
\newblock Square root time {C}oleman integration on superelliptic curves.
\newblock {\em Arithmetic geometry, number theory, and computation, Simons Symposia}, 2023.

\bibitem[BMS16]{BMS16}
Jennifer~S. Balakrishnan, J.~Steffen M\"{u}ller, and William~A. Stein.
\newblock A {$p$}-adic analogue of the conjecture of {B}irch and {S}winnerton-{D}yer for modular abelian varieties.
\newblock {\em Math. Comp.}, 85(298):983--1016, 2016.

\bibitem[BMS21]{BMS21}
Amnon Besser, J.~Steffen Müller, and Padmavathi Srinivasan.
\newblock $p$-adic adelic metrics and {Q}uadratic {C}habauty {I}.
\newblock {\em Arxiv preprint, \url{https://arxiv.org/abs/2112.03873}}, 2021.

\bibitem[BPR93]{BPR93}
Dominique Bernardi and Bernadette Perrin-Riou.
\newblock Variante {$p$}-adique de la conjecture de {B}irch et {S}winnerton-{D}yer (le cas supersingulier).
\newblock {\em C. R. Acad. Sci. Paris S\'{e}r. I Math.}, 317(3):227--232, 1993.

\bibitem[BS10]{Bruin-Stoll-MW-Sieve}
Nils Bruin and Michael Stoll.
\newblock The {M}ordell–{W}eil sieve: proving non-existence of rational points on curves.
\newblock {\em LMS Journal of Computation and Mathematics}, 13:272–306, 2010.

\bibitem[BT20]{BalakrishnanTuitman}
Jennifer~S. Balakrishnan and Jan Tuitman.
\newblock Explicit {C}oleman integration for curves.
\newblock {\em Math. Comp.}, 89(326):2965--2984, 2020.

\bibitem[CdS88]{Coleman-deShalit}
Robert~F. Coleman and Ehud de~Shalit.
\newblock {$p$}-adic regulators on curves and special values of {$p$}-adic {$L$}-functions.
\newblock {\em Invent. Math.}, 93(2):239--266, 1988.

\bibitem[CG89]{Coleman-Gross-Heights}
Robert~F. Coleman and Benedict~H. Gross.
\newblock {$p$}-adic heights on curves.
\newblock In {\em Algebraic number theory}, volume~17 of {\em Adv. Stud. Pure Math.}, pages 73--81. Academic Press, Boston, MA, 1989.

\bibitem[Col85]{ColemanI}
Robert~F. Coleman.
\newblock Torsion points on curves and $p$-adic abelian integrals.
\newblock {\em Annals of Mathematics}, 121(1):111--168, 1985.

\bibitem[Col91]{Coleman-bi-extension-p-adic-heights}
Robert~F. Coleman.
\newblock The universal vectorial bi-extension and {$p$}-adic heights.
\newblock {\em Invent. Math.}, 103(3):631--650, 1991.

\bibitem[Col98]{Col98}
Pierre Colmez.
\newblock Int\'{e}gration sur les vari\'{e}t\'{e}s {$p$}-adiques.
\newblock {\em Ast\'{e}risque}, (248):viii+155, 1998.

\bibitem[DRHS23]{DRHS}
Juanita Duque-Rosero, Sachi Hashimoto, and Pim Spelier.
\newblock {G}eometric quadratic {C}habauty and $p$-adic heights.
\newblock {\em Expositiones Mathematicae, to appear}, 2023.

\bibitem[EL23]{EL19}
Bas Edixhoven and Guido Lido.
\newblock Geometric quadratic {C}habauty.
\newblock {\em J. Inst. Math. Jussieu}, 22(1):279--333, 2023.

\bibitem[FLS{\etalchar{+}}01]{FLSSSW01}
E.~Victor Flynn, Franck Lepr\'{e}vost, Edward~F. Schaefer, William~A. Stein, Michael Stoll, and Joseph~L. Wetherell.
\newblock Empirical evidence for the {B}irch and {S}winnerton-{D}yer conjectures for modular {J}acobians of genus 2 curves.
\newblock {\em Math. Comp.}, 70(236):1675--1697, 2001.

\bibitem[Gaj22]{Gajovic-thesis}
Stevan Gajovi\'c.
\newblock {\em Variations on the method of {C}habauty and {C}oleman}.
\newblock PhD thesis, University of Groningen, 2022.

\bibitem[GM23]{LinQC}
Stevan Gajovi\'c and J.~Steffen Müller.
\newblock {L}inear quadratic {C}habauty.
\newblock {\em ArXiv preprint, \url{https://arxiv.org/abs/2307.15781}, to appear in Isreal Journal of Mathematics}, 2023.

\bibitem[Gro86]{Gross-local-heights}
Benedict~H. Gross.
\newblock Local heights on curves.
\newblock In {\em Arithmetic geometry ({S}torrs, {C}onn., 1984)}, pages 327--339. Springer, New York, 1986.

\bibitem[Har12]{Harrison-Even-Deg-MW}
Michael~C. Harrison.
\newblock An extension of {K}edlaya's algorithm for hyperelliptic curves.
\newblock {\em J. Symbolic Comput.}, 47(1):89--101, 2012.

\bibitem[Hol12]{Hol12}
David Holmes.
\newblock Computing {N}\'eron-{T}ate heights of points on hyperelliptic {J}acobians.
\newblock {\em J. Number Theory}, 132(6):1295--1305, 2012.

\bibitem[Iov00]{Iov00}
Adrian Iovita.
\newblock Formal sections and de {R}ham cohomology of semistable abelian varieties.
\newblock In {\em Proceedings of the {C}onference on {$p$}-adic {A}spects of the {T}heory of {A}utomorphic {R}epresentations ({J}erusalem, 1998)}, volume 120, pages 429--447, 2000.

\bibitem[Ked01]{Kedlaya-MW-reduction}
Kiran~S. Kedlaya.
\newblock Counting points on hyperelliptic curves using {M}onsky-{W}ashnitzer cohomology.
\newblock {\em J. Ramanujan Math. Soc.}, 16(4):323--338, 2001.

\bibitem[Kim05]{Kim2005MFG}
Minhyong Kim.
\newblock The motivic fundamental group of $\mathbb{P}^1\backslash\{0,1,\infty\}$ and the theorem of {S}iegel.
\newblock {\em Inventiones mathematicae}, 161(3):629--656, 2005.

\bibitem[Kim09]{Kim2009Selmer}
Minhyong Kim.
\newblock The unipotent {A}lbanese map and {S}elmer varieties for curves.
\newblock {\em Publ. Res. Inst. Math. Sci.}, 45(1):89--133, 2009.

\bibitem[KS22]{KS22}
Timo Keller and Michael Stoll.
\newblock Exact verification of the strong {BSD} conjecture for some absolutely simple abelian surfaces.
\newblock {\em C. R. Math. Acad. Sci. Paris}, 360:483--489, 2022.

\bibitem[KS23]{KS}
Timo Keller and Michael Stoll.
\newblock Complete verification of strong bsd for many modular abelian surfaces over $\mathbf{Q}$.
\newblock {\em ArXiv preprint, \url{https://arxiv.org/abs/2312.07307}}, 2023.

\bibitem[MSD74]{MSD74}
Barry Mazur and Peter Swinnerton-Dyer.
\newblock Arithmetic of {W}eil curves.
\newblock {\em Invent. Math.}, 25:1--61, 1974.

\bibitem[MT83]{Mazur-Tate-p-adic-heights}
Barry Mazur and John Tate.
\newblock Canonical height pairings via biextensions.
\newblock In {\em Arithmetic and geometry, {V}ol. {I}}, volume~35 of {\em Progr. Math.}, pages 195--237. Birkh\"{a}user Boston, Boston, MA, 1983.

\bibitem[MTT86]{MTT86}
Barry Mazur, John Tate, and Jeremy Teitelbaum.
\newblock On {$p$}-adic analogues of the conjectures of {B}irch and {S}winnerton-{D}yer.
\newblock {\em Invent. Math.}, 84(1):1--48, 1986.

\bibitem[M{\"u}l14]{Mul14}
J.~Steffen M{\"u}ller.
\newblock Computing canonical heights using arithmetic intersection theory.
\newblock {\em Math. Comp.}, 83(285):311--336, 2014.

\bibitem[Nek93]{Nek93}
Jan Nekov\'{a}\v{r}.
\newblock On {$p$}-adic height pairings.
\newblock In {\em S\'{e}minaire de {T}h\'{e}orie des {N}ombres, {P}aris, 1990--91}, volume 108 of {\em Progr. Math.}, pages 127--202. Birkh\"{a}user Boston, Boston, MA, 1993.

\bibitem[PS11]{PS11}
Robert Pollack and Glenn Stevens.
\newblock Overconvergent modular symbols and {$p$}-adic {$L$}-functions.
\newblock {\em Ann. Sci. \'{E}c. Norm. Sup\'{e}r. (4)}, 44(1):1--42, 2011.

\bibitem[Sch82]{Schneider-p-adic-heights1}
Peter Schneider.
\newblock {$p$}-adic height pairings. {I}.
\newblock {\em Invent. Math.}, 69(3):401--409, 1982.

\bibitem[SW13]{SW13}
William Stein and Christian Wuthrich.
\newblock Algorithms for the arithmetic of elliptic curves using {I}wasawa theory.
\newblock {\em Math. Comp.}, 82(283):1757--1792, 2013.

\bibitem[{The}23]{Sage}
{The Sage Developers}.
\newblock {\em {S}ageMath, the {S}age {M}athematics {S}oftware {S}ystem ({V}ersion 9.8)}, 2023.
\newblock {\tt https://www.sagemath.org}.

\bibitem[Tui17]{Tuitman-P1-two}
Jan Tuitman.
\newblock Counting points on curves using a map to {$\bold{P}^1$}, {II}.
\newblock {\em Finite Fields Appl.}, 45:301--322, 2017.

\bibitem[vB22]{vB22}
Raymond van Bommel.
\newblock Numerical verification of the {B}irch and {S}winnerton-{D}yer conjecture for hyperelliptic curves of higher genus over {$\Bbb{Q}$} up to squares.
\newblock {\em Exp. Math.}, 31(1):138--145, 2022.

\bibitem[vBHM20]{Raymond-David-Steffen}
Raymond van Bommel, David Holmes, and J.~Steffen M\"{u}ller.
\newblock Explicit arithmetic intersection theory and computation of {N}\'{e}ron-{T}ate heights.
\newblock {\em Math. Comp.}, 89(321):395--410, 2020.

\bibitem[vdP86]{Marius-MW-Cohomology}
Marius van~der Put.
\newblock The cohomology of {M}onsky and {W}ashnitzer.
\newblock Number~23, pages 4, 33--59. 1986.
\newblock Introductions aux cohomologies $p$-adiques (Luminy, 1984).

\end{thebibliography}

\end{document}